\documentclass[11pt]{article}

\let\oldmarginpar\marginpar
\renewcommand\marginpar[1]{\-\oldmarginpar[\raggedleft\footnotesize #1]%
{\raggedright\footnotesize #1}}

\usepackage{mathptmx}
\usepackage{amsmath}
\usepackage{amscd}
\usepackage{amssymb}
\usepackage{amsthm}
\usepackage{xspace}
\usepackage[all,tips]{xy}
\usepackage[dvips]{graphicx}
\usepackage{verbatim}
\usepackage{syntonly}
\usepackage{hyperref}
\usepackage{arydshln}
\usepackage{lineno,color}
\usepackage{todonotes}
\usepackage{faktor}

\theoremstyle{plain}
\newtheorem{thm}{Theorem}[section]
\newtheorem{cor}[thm]{Corollary}
\newtheorem{prop}[thm]{Proposition}
\newtheorem{lemma}[thm]{Lemma}
\newtheorem{ques}{Question}
\newtheorem{conj}{Conjecture}

\theoremstyle{definition}

\newtheorem{defn}[thm]{Definition}

\newtheorem{ex}[thm]{Example}

\newtheoremstyle{TheoremNum}
{\topsep}{\topsep}
{\itshape}
{}
{\bfseries}
{.}
{ }
{\thmname{#1}\thmnote{ \bfseries #3}}
\theoremstyle{TheoremNum}
\newtheorem{thmn}{Theorem}

\DeclareMathOperator{\Aut}{Aut}

\DeclareMathOperator{\GL}{GL}

\DeclareMathOperator{\GCD}{gcd}

\DeclareMathOperator{\D}{D}

\DeclareMathOperator{\Farb}{F}
\DeclareMathOperator{\Conj}{Conj}\DeclareMathOperator{\Inn}{Inn}

\DeclareMathOperator{\Tconj}{Tconj}

\DeclareMathOperator{\Ord}{ord}
\DeclareMathOperator{\im}{Im}
\DeclareMathOperator{\id}{id}

\DeclareMathOperator{\rank}{rank}
\DeclareMathOperator{\ab}{ab}

\newcommand{\suchthat}{\;\ifnum\currentgrouptype=16 \middle\fi|\;}

\newcommand{\bdef}{\overset{\text{def}}{=}}

\newcommand{\al}{\alpha}

\newcommand{\ga}{\gamma}
\newcommand{\Ga}{\Gamma}

\newcommand{\innp}[1]{\left< #1 \right>}

\newcommand{\set}[1]{\left\{#1\right\}}

\newcommand{\pr}[1]{\left( #1 \right) }

\newcommand{\N}{\ensuremath{\mathbb{N}}}

\newcommand{\Z}{\ensuremath{\mathbb{Z}}}

\newcommand{\map}[3]{#1 : #2 \rightarrow #3}
\newcommand{\surj}[3]{#1: #2 \twoheadrightarrow #3}

\newcommand{\nsub}{\trianglelefteq}
\usepackage{fancyhdr}

\fancypagestyle{plain}

\begin{document}
\title{\textbf{Effective Twisted Conjugacy \\ Separability of Nilpotent Groups}}
\author{Jonas Der\'e\thanks{KU Leuven Kulak, Kortrijk, Belgium. Email: \tt{jonas.dere@kuleuven.be}. Supported by a postdoctoral fellowship of the Research Foundation -- Flanders (FWO).} \ and Mark Pengitore\thanks{Purdue University, West Lafayette, IN. E-mail: \tt{mpengito@purdue.edu}}} 
\maketitle
\begin{abstract}
This paper initiates the study of effective twisted conjugacy separability for finitely generated groups, which measures the complexity of separating distinct twisted conjugacy classes via finite quotients. The focus is on nilpotent groups, and our main result shows that there is a polynomial upper bound for twisted conjugacy separability. That allows us to study regular conjugacy separability in the case of virtually nilpotent groups, where we compute a polynomial upper bound as well. As another application, we improve the work of the second author by giving a possibly sharp upper bound for the conjugacy separability for finitely generated nilpotent groups of nilpotency class $2$.
\end{abstract}

\section{Introduction}
\label{intro}
Let $G$ be a finitely generated group with a fixed automorphism $\varphi$. We say that $x,y \in G$ are \emph{$\varphi$-twisted conjugate} if there exists $z \in G$ such that $z \: x \: \varphi(z)^{-1} = y$. That is an equivalence relation and the equivalence classes of $x \in G$ is called the $\varphi$-twisted conjugacy class which is denoted as $[x]_{\varphi}$. We say $G$ is \emph{$\varphi$-twisted conjugacy separable} if for each non-$\varphi$-twisted conjugate pair $x,y \in G$, there exists a surjective group morphism to a finite group $\map{\pi}{G}{Q}$ such that $\pi(x) \notin \pi([y]_\varphi)$, or equivalently, $\pi(y) \notin \pi([x]_\varphi)$. When $G$ is $\varphi$-twisted conjugacy separable for all $\varphi \in \Aut(G)$, we say that $G$ is \emph{twisted conjugacy separable}.

The interest in twisted conjugacy classes arises in many different areas of mathematics such as Reidemeister fixed point theory \cite{Felshtyn_zeta,Felshtyn_nielsen_reidemeister_theory,Jiang}, Selberg theory \cite{Arthur_Clozel,Shokranian}, and algebraic geometry \cite{Grothendieck}. Twisted conjugacy separability was originally introduced formally in \cite{Felstyn_2} and has been further studied in \cite{Felstyn_1}.

When $\varphi = \id$, the $\varphi$-twisted conjugacy classes are equal to the usual notion of conjugacy classes. Similarly, $\varphi$-twisted conjugacy separable is equivalent to the usual notion of conjugacy separable when $\varphi = \id$ . Thus, the notion of $\varphi$-twisted conjugacy separability is a natural generalization of conjugacy separability where we allow our conjugacy classes to be twisted by $\varphi$. In particular, if $G$ is a finitely generated group that contains a characteristic, finite index, twisted conjugacy separable subgroup, then $G$ is twisted conjugacy separable, see \cite[Thm 5.2]{Felstyn_1}. That is in contrast to conjugacy separability which is not closed with respect to finite extensions or finite index subgroups \cite{Goryaga,Martino_Minasyan}. Thus, twisted conjugacy separability gives us an important tool in studying the more classical notion of conjugacy separability since twisted conjugacy separability implies conjugacy separability.

Conjugacy separability, along with residually finiteness, subgroup separability, and other residual properties, have been extensively studied and used in resolving important conjectures in geometry such as Agol's work on the Virtual Haken conjecture. Previous work in the literature has been to understand what groups satisfy these properties. For instance, polycyclic groups, free groups, residually free groups, limit groups, Bianchi groups, surface groups, fundamental groups of compact, orientable $3$-manifolds, and virtually compact special hyperbolic groups have been shown to be conjugacy separable and subsequently, residually finite \cite{Blackburn,Chagas_Zalesskii_bianchi_groups,Chagas_Zalesskii,Chagas_Zalesskii_limit_groups,Formanek,conjugacy_3_manifolds,Chagas_Zalesskii_virtually_special,Stebe}. Recently, there has been considerable activity in establishing effective versions of the above separability properties (see \cite{Bou_Rabee_Mcreynolds1,Bou_Rabee_Mcreynolds2,Bou_Rabee10,Bou_Rabee11,Bou_Rabee_2015,Bou_Rabee_Kaletha,Bou_Rabee_Mcreynolds3,Bou_Rabee_Seward,Buskin,Kassabov_Matucci,Kharlampovich_Myasnikov_Sapir,Kozma_Thom,Patel,Patel2,Rivin,Thom}). The main purpose of this article is to improve on the effective conjugacy separability results of \cite{Pengitore_1} for nilpotent groups and to establish effective twisted conjugacy separability for the class of virtual nilpotent groups.

For a finitely generated group $G$ with a finite generating subset $S$ and an automorphism $\varphi$, we introduce a function $\Conj_{G,S}^\varphi(n)$ on the natural numbers that quantifies $\varphi$-twisted conjugacy separability. To be specific, the value on a natural number $n$ is the maximum order of the minimal finite quotient needed to distinguish pairs of non-$\varphi$-conjugate elements as one varies over the $n$-ball. We also quantify the more general property of twisted conjugacy separability via the function $\Tconj_{G,S}(n)$. We start by establishing a norm $\|\varphi\|_S$ for automorphisms and then define $\Tconj_{G,S}(n)$ on a natural number $n$ to be the maximum value of $\Conj_{G,S}^\varphi(n)$ as one varies over automorphisms $\varphi$ satisfying $\|\varphi\|_S \leq n$. 

When the automorphism $\varphi$ is the identity map, the function $\Conj_{G,S}^\varphi$ is equal to the function that quantifies conjugacy separability introduced by Lawton-Louder-McReynolds \cite{LLM}, which is in this case denoted as $\Conj_{G,S}(n)$. As a natural consequence, $\Conj_{G,S}(n)$ is always bounded by $\Tconj_{G,S}(n)$. Additionally, we will see that if $H$ is a characteristic, twisted conjugacy separable subgroup of $G$ of index $r$, then there exist automorphisms $\set{f_i}_{i=1}^r$ of $H$ such that $\Conj_{G,S}(n)$ is bounded by $\prod_{i=1}^r(\Conj_{H,S'}^{f_i}(n))^{r}$ where  $S$ and $S^\prime$ are finite generating subsets of $G$ and $H$, respectively (see Theorem \ref{ext_twisted_conj}). Hence, if we are given a class of conjugacy separable groups that is closed under finite index subgroups, then we may study conjugacy separability of any finite extension of these groups by studying effective twisted conjugacy separability of the original group.

So far, previous papers have only studied the asymptotic behavior of $\Conj_{G,S}(n)$. For instance, if $G$ is a finitely generated nilpotent group, then the second author \cite{Pengitore_1} demonstrated $\Conj_{G,S}(n) \preceq n^{d_1}$ for some $d_1 \in \N$, and if, in addition, $G$ is a finitely generated nilpotent group that is not virtually abelian, then there exists $d_2 \in \N$ such that $n^{d_2} \preceq \Conj_{G,S}(n)$. Lawton-Louder-McReynolds demonstrated that $\Conj_{G,S}(n) \preceq n^{n^2}$ when $G$ is a finite rank free group or a surface group. This article is the first to study effective twisted conjugacy separability for any class of groups. 
As of now, there are no effective proofs of conjugacy separability for classes of groups such as polycyclic groups, fundamental groups of compact, orientable $3$-manifolds, Bianchi groups, etc. Even in the context of surface groups or free groups, there is no good asymptotic lower bound for $\Conj_{G,S}(n)$ other than the one provided by effective residual finiteness.

To state our results, we require some notation. For two non-decreasing functions $\map{f,g}{\N}{\N}$, we write $f(n) \preceq g(n)$ if there exists a $C \in \N$ such that $f(n) \leq C \: g(C\:n)$ for all $n$. We write $f \approx g$ when $f \preceq g$ and $g \preceq f$. If $N$ is a nilpotent group, we write $c(N)$ for its nilpotent class, as we will recall in Section \ref{sec:nilpotent}.

Our first result is the presumably exact upper bound for $\Conj_{N,S}(n)$ when $N$ is a $2$-step, torsion free, finitely generated nilpotent group and $S$ is a finite generating subset. 
\begin{thmn}[\ref{two_step_conjugacy_precise}]
	Let $N$ be a torsion free, finite generated nilpotent group with a finite generating subset $S$ such that $c(N) = 2$. Then there exists $\Phi(N) \in \N$ such that $\Conj_{N,S}(n) \preceq n^{\Phi(N)}$.
\end{thmn}
An explicit expression for $\Phi(N)$ is given in Section \ref{sec:nilpotent}. This matches the lower bound provided by \cite[Thm 1.8]{Pengitore_1}, but unfortunately we were recently made aware that there is a gap in the proof of that paper. It is still likely that this result gives a sharp upper bound, but at the moment there is no full proof for the lower bound. The purpose of this result is to both improve the results of \cite{Pengitore_1} in the context of two step nilpotent groups and to demonstrate the techniques used in the proof of the main theorem.

The next result is the main theorem of the article. In the following theorem, we give the first effective upper bound for $\varphi$-twisted conjugacy separability for nilpotent groups and the first effective upper bound for twisted conjugacy separability for any finitely generated group.
\begin{thmn}[\ref{main_thm}]
	Let $N$ be a torsion free, finitely generated nilpotent group with a finite generating subset $S$. Let $x \in N$ and $\varphi \in \Aut(N)$. Then there exist natural numbers $k_1,k_2,k_3$ such that
	$$
	\Conj_{N,x,S}^\varphi(n) \preceq \pr{\|\varphi\|_S}^{k_1} \: \pr{\|x\|_S}^{k_2} \: n^{k_3}.
	$$
	In particular, $\Conj_{N,S}^\varphi(n) \preceq \pr{\|\varphi\|_S}^{k_1} n^{k_2 + k_3}$ and $\Tconj_{N,S}(n) \preceq n^{k_1 + k_2 + k_3}$.
\end{thmn}
The proof of Theorem \ref{main_thm} generalizes the techniques of \cite{Blackburn} and \cite{Pengitore_1} to the context of $\varphi$-conjugacy classes by introducing the notion of the $i$-th twisted centralizer corresponding to an automorphism $\varphi$. When $\varphi = \Inn(x)$, then the $i$-th twisted centralizer is equal to the group of elements that centralize $x$ modulo the $i$-th term of the lower central series. Using these twisted centralizers, we may proceed by induction on step length.

More generally, we have a similar result as Theorem \ref{main_thm} for virtually nilpotent groups.
\begin{thmn}[\ref{virtual_upper_bound}]
Suppose that $G$ is a virtually nilpotent group, and suppose that $S$ is a finite generating subset of $G$. For $\varphi \in \Aut(G)$ and $x \in G$, there exist natural numbers $k_1,k_2,k_3$ such that
$$
\Conj_{G,x,S}(n) \preceq \pr{\|\varphi\|_S}^{k_1} \: \pr{\|x\|_S}^{k_2} \: n^{k_3}.
$$
In particular, $\Conj_{G,S}^\varphi(n) \preceq \pr{\|\varphi\|_S}^{k_1} n^{k_2 + k_3}$ and $\Tconj_{G,S}(n) \preceq n^{k_1 + k_2 + ks_3}$.
\end{thmn}
For this result, we write the $\varphi$-twisted conjugacy class of an element of $G$ as a finite union of right translates of twisted conjugacy classes of a finite index, characteristic finitely generated nilpotent subgroup. We then apply Theorem \ref{main_thm}.

The last result of this article extends the work of the second author in \cite{Pengitore_1} to the context of finite extensions of nilpotent groups.
\begin{thmn}[\ref{last_main_result}]
	Let $G$ be a virtually nilpotent group with a finite generating subset $S$, and let $x \in G$. There exist $k_1, k_2 \in \N$ such that $\Conj_{G,x,S}(n) \preceq \pr{\|x\|_S}^{k_1} n^{k_2}$. In particular, $\Conj_{G,S}(n) \preceq n^{k_1 + k_2}$. If $G$ is not virtually abelian, then
	$$
	n^{(c(N) - 1) \: (c(N) + 1)} \preceq \Conj_{G,S}(n) \preceq n^{k_1 + k_2}
	$$
	where $N$ is any infinite finitely generated nilpotent subgroup of finite index in $G$.
\end{thmn}
In order to prove the upper bound of Theorem \ref{last_main_result}, we apply Theorem \ref{virtual_upper_bound}. For the lower bound, we follow the methods of \cite[Thm 1.8]{Pengitore_1} using the fact that every conjugacy class in $G$ is the union of at most $[G:N]$ conjugacy classes in $N$.

We finish by working out asymptotic upper bounds for $\Conj_{H_3(\Z)}^\varphi(n)$ for any $\varphi \in \Aut(H_3(\Z))$ where $H_3(\Z)$ is the $3$-dimensional integral Heisenberg group. We also work out a $5$-dimension example with a fixed automorphism.
\paragraph{Acknowledgements}

We would like to thank Karel Dekimpe for his suggestion to study twisted conjugacy classes in the context of finitely generated nilpotent groups. The second author would like to thank his advisor Ben McReynolds for his continued support.
\section{Background}

In this section, we introduce necessary definitions for this paper and start by fixing some notation. 

Let $G$ be a group with finite generating subset $S$. The order of a finite group $G$ is denoted as $\vert G \vert$. We write $\|x\|_S$ to be the word length of $x$ with respect to $S$ and denote the identity element of $G$ as $1$. We denote the order of $g$ as an element of the group $G$ as $\Ord_G(g)$. We define the commutator of $x, y \in G$ as $[x,y] = x \: y \: x^{-1} \: y^{-1}$. For any subset $X \subset G$, we let $\langle X \rangle$ be the subgroup generated by the set $X$. For a normal subgroup $H \nsub G$, we set $\map{\pi_H}{G}{G/H}$ to be the natural projection and sometimes write $\bar{x} = \pi_H(x)$ if the normal subgroup $H$ is clear from the context. 

For any $x \in G$, we let $\Inn(x) \in \Aut(G)$ be the associated inner automorphism i.\@e.\@ $\Inn(x)(y) = x \: y \: x^{-1}.$ For an integer $k$ and prime $p$, we define $v_p(k)$ to be the largest natural number such that $p^{v_p(k)}$ divides $k$. 

We define $G^m$ to be the subgroup generated by $m$-th powers of elements in $G$ which is a characteristic subgroup. We define the associated projection as $r_m = \pi_{G^m}$.  We also define the abelianization of $G$ as $G_\text{ab} = G / [G,G]$ with the associated projection $\pi_{\text{ab}} = \pi_{[G,G]}$. We define  the center of $G$ as $Z(G)$ and the centralizer of $x$ in $G$ as $C_G(x)$.

\subsection{Norms of subgroups and automorphisms}

We associate a norm to finitely generated subgroups in the following definition.
\begin{defn}
	\label{def_normsubgroup}
	Let $G$ be a finitely generated group with a finite generating subset $S$. For any finite subset $X \subseteq G$, we define $\|X\|_S = \max\{\|x\|_S \: | \: x \in X\}$. For a finitely generated subgroup $H \leq G$, we define
	$$
	\|H\|_S = \text{min}\{\|X\|_S \: | \: X \text{ is a finite generating subset for } H \}.
	$$
\end{defn}

Let $S_1$ and $S_2$ be two generating subsets of a group $G$ such thath $\vert s \vert_{S_2} \leq K$ for all $s \in S_1$, or equivalently, such that $\Vert x \Vert_{S_2} \leq K \Vert x \Vert_{S_1}$ for all $x \in G$. Then $\Vert H \Vert_{S_2} \leq K \Vert H \Vert_{S_1}$ for all finitely generated subgroups $H \le G$.
Indeed, take generators $h_i \in H$ for $H$ such that $\vert h_i \vert_{S_1} \leq \Vert H \Vert_{S_1}$, then $\vert h_i \vert_{S_2} \leq K \vert h_i \vert_{S_1} \leq K \Vert H \Vert_{S_1}$. Since these elements generate $H$, the statement follows.
In particular, we get the following relation between norms of subgroups for different generating subsets.

\begin{lemma}
	Let $G$ be a finitely generated group with finite generating subsets $S_1$ and $S_2$, and let $H \leq G$ be a finitely generated subgroup. Then $\frac{1}{C} \:  \| H \|_{S_1} \leq \|H\|_{S_2} \leq C \: \|H \|_{S_1}$ for some $C > 0$.
\end{lemma}

Similarly, we define a norm for morphisms of finitely generated groups, and subsequently, define a norm for automorphisms of a finitely generated group. 

\begin{defn}
	\label{def_normauto}
	Let $G$ and $G'$ be finitely generated groups with finite generating subsets $S$ and $S'$. Let $\map{\varphi}{G}{G'}$ be a group morphism. We define
	$$
	\| \varphi\|_{S,S'} = \max\{\|\varphi(s) \|_{S'} \: | \: s \in S \}.
	$$
	If $\varphi$ is an automorphism of $G$, then we assume $S = S'$ and write $\| \varphi \|_S$.
\end{defn}

Equivalently, $\Vert \varphi \Vert_{S,S^\prime}$ is the smallest natural number such that $\Vert \varphi(x) \Vert_{S^\prime} \leq \Vert \varphi \Vert_{S, S^\prime} \Vert x \Vert_S$ for all $x \in G$.

Let $S_1$ and $S_2$ be two generating subsets of $G$ such that $\Vert s \Vert_{S_2} \leq K$ for all $s \in S_1$. We show that in this case, $\Vert \varphi \Vert_{S_1,S^\prime} \leq K \Vert \varphi \Vert_{S_2,S^\prime}$. Note that $$\Vert \varphi(x) \Vert_{S^\prime} \leq \Vert \varphi \Vert_{S_2, S^\prime} \Vert x \Vert_{S_2} \leq K \Vert \varphi \Vert_{S_2,S^\prime} \Vert x \Vert_{S_1}$$ for all $x \in G$, and thus, the conclusion follows. 

Similarly, one can show for $S_1^\prime$ and $S_2^\prime$ generating subsets for $G^\prime$ with $\Vert s \Vert_{S_2^\prime} \leq K$ for all $s \in S_1^\prime$, that $$\Vert \varphi \Vert_{S,S_1^\prime} \leq K \Vert \varphi \Vert_{S,S_2^\prime}.$$

In particular, we have the following lemma.

\begin{lemma}Let $G$ be a group with two finite generating subsets $S, S^\prime$. There exists a constant $C > 0$ such that for every automorphism $\varphi \in \Aut(G)$, we have $\frac{1}{C} \Vert \varphi \Vert_{S^\prime} \leq \Vert \varphi \Vert_S \leq C \Vert \varphi \Vert_{S^\prime}.$
\end{lemma}

\subsection{Finitely generated groups and separability}
Let $G$ be a finitely generated group with a finite generating subset $S$, and let $X \subseteq G$ be an arbitrary subset. Following Bou-Rabee in \cite{Bou_Rabee10}, we define the relative depth function $$\map{D_G(X, \cdot \hspace{1mm})}{G \setminus   X}{\N \cup \{\infty\}}$$ as $$ \label{def_reldepfun}
\D_{G}(X,y) = \text{min}\{|Q| \suchthat \surj{\pi}{G}{Q}, \hspace{1mm} |Q| < \infty, \text{ and } \pi(y) \notin \pi(X) \}
$$
with the understanding that $D_G(X,y) = \infty$ if no such $Q$ exists.
\begin{defn}
We say that a non-empty subset $X \leq G$ is \emph{separable} if $D(X,y) < \infty$ for all $y \in G \setminus  X$. We say that a finite quotient $G/H$ \emph{separates} $X$ and $y$ if $\pi_H(y) \notin \pi_H(X)$.
\end{defn}

Recall the twisted conjugacy class $[x]_\varphi$ for $x \in G$ and $\varphi \in \Aut(G)$ from Section \ref{intro}.
\begin{defn}
	Let $G$ be a finitely generated group with and let $\map{\varphi}{G}{G}$ be an automorphism. We say that $G$ is \emph{$\varphi$-twisted conjugacy separable} if for all $x \in G$ the twisted conjugacy class $[x]_\varphi$ is separable. We say that $G$ is \emph{twisted conjugacy separable} if $G$ is $\varphi$-twisted conjugacy separable for all $\varphi \in \Aut(G)$.
\end{defn}

Let $G$ be a finitely generated group with a finite generating subset $S$, and let $\map{\varphi}{G}{G}$ be an automorphism. To quantify $\varphi$-twisted conjugacy separability relative to $x$ for some fixed $x \in N$, we define the following function $\map{\Conj_{G,x,S}^\varphi(n)}{\N}{\N}$ given by
$$
\Conj_{G,x,S}^{\varphi}(n) = \max\{ D([x]_\varphi,y) \: | \: \|y\|_S \leq n, y \notin [x]_\varphi \}.
$$
By allowing $x$ to vary, we are able to quantify $\varphi$-twisted conjugacy separability for any given group $G$ and automorphism $\varphi$. We define $\map{\Conj^\varphi_{G,S}}{\N}{\N}$ to be given by $$\Conj^\varphi_{G,S}(n) = \max\{ \Conj_{G,x,S}^\varphi(n) \: | \: \|x\|_S \leq n \}. $$

Finally, we obtain a method to quantify twisted conjugacy separability by taking automorphisms with norm at most $n$. We now define $\map{\Tconj}{\N}{\N}$ to be given by $$\Tconj_{G,S}(n) =  \max \{\Conj_{G,S}^\varphi(n) \: | \: \varphi \in \Aut(G) \text{ and } \|\varphi\|_S \leq n \}.$$

\begin{lemma}
	Let $S_1$ and $S_2$ be two finite generating subsets of $G$. If $\varphi \in \Aut(G)$, then $\Conj_{G,S_1}^\varphi(n) \approx \Conj_{G,S_2}^\varphi(n)$. Similarly, $\Tconj_{G,S_1}(n) \approx \Tconj_{G,S_2}(n)$. 
\end{lemma}
The proof is similar to \cite[Lem 1.1]{Bou_Rabee10} (see also \cite[Lem 1.2]{Bou_Rabee11} and \cite[Lem 2.1]{LLM}).

Observe that if $\varphi = \id$, then $\Conj_{G,S}^\varphi$ is equal to the conjugacy separability $\Conj_{G,S}$ introduced by Lawton, Louder, and McReynolds \cite{LLM}; subsequently, $\Conj_{G,S}(n) \preceq \Tconj_{G,S}(n)$. If, for some $\varphi \in \Aut(G)$, we have that $\Conj_{G,S}^\varphi(n) < \infty$ for all $n \in \N$, then $G$ is $\varphi$-twisted conjugacy separable. Similarly, if $\Tconj_{G,S}(n) < \infty$ for all $n$, then $G$ is twisted conjugacy separable, and subsequently, $G$ is $\varphi$-conjugacy separable for all $\varphi \in \Aut(G)$. In particular, $G$ is conjugacy separable in that case.

\subsection{Nilpotent Groups}
\label{sec:nilpotent}
Most of the groups we work on in this paper are nilpotent groups, and we recall their basic properties. See \cite{Hall_notes,Segal_book_polycyclic} for a more thorough account of the theory of nilpotent groups.

A central series for a group $N$ is a sequence of subgroups $1 = N_0 \le N_1 \le \dots \le N_c = N$ such that $[N,N_{i+1}] \le N_i$. There are two special central series which play an important role when studying nilpotent groups. The $i$-th term of the \emph{lower central series} is defined by $\ga_1(N) \bdef G$ and inductively $\ga_i(N) \bdef [G,\ga_{i-1}(N)]$ for $i>1$.  The $i$-th term of the \emph{upper central series} is defined by $\zeta_1(N) \bdef Z(N)$ and inductively $\zeta_i(N) \bdef \pi_{\zeta_{i-1}(N)}^{-1}(Z(G / \zeta_{i-1}(N)))$ for $i>1$. We denote their associated projections as $\pi_i = \pi_{\ga_i(N)}$ and $\pi^i = \pi_{\zeta_i(N)}$ when $N$ is understood from context.

\begin{defn}
	We say that $N$ is \emph{nilpotent of step size $c$} if $c$ is the minimal natural number such that $\ga_{c + 1}(N) = 1$, or equivalently, $\zeta_{c}(N) = N$. We write the nilpotency class of $N$ as $c(N)$. We define the \emph{Hirsch length} of $N$ as $$h(N) = \sum_{i=1}^{c(N)}\text{rank}_\Z \left( \faktor{\ga_i(N)}{\ga_{i+1}(N)}\right).$$ We define $T(N)$ to be the finite characteristic subgroup of finite order elements. If $N$ is a torsion free, finitely generated nilpotent group, we say that $N$ is a $\mathcal{F}$-group. 
\end{defn}
For nilpotent groups $N$, the subgroup $N^m$ is always of finite index in $N$. 

\begin{lemma}\label{congruence_subgroup_of_finite_index_subgroup}
	Let $N$ be a $\mathcal{F}$-group, and let $H \leq N$ be a subgroup of index $p^\al$ where $p$ is a prime. Then $N^{p^\al} \leq H$.
\end{lemma}
\begin{proof}
	For normal subgroups $H \le N$, this result is trivial. Using the fact that in nilpotent groups every subgroup $H$ in $N$ is subnormal, see \cite{Segal_book_polycyclic}, meaning that there exists a sequence $H_0 = H \le H_1 \le \ldots \le H_k = N$ such that $H_i$ is normal in $H_{i+1}$, the result follows for every subgroup $H$.
\end{proof}

The following subgroup will be useful in assigning numerical invariants to subgroups and group morphisms.
\begin{defn}
	Let $N$ be a $\mathcal{F}$-group, and let $H \leq N$ be a subgroup. We define the \emph{isolator of $H$ in $N$}, denoted $\sqrt[N]{H}$, as the set
	$$
	\sqrt[N]{H} = \{ x \in N \: | \: \text{there exists } k \in \N \text{ such that } x^k \in H \}.
	$$
\end{defn}
From \cite{Segal_book_polycyclic} it follows that $\sqrt[N]{H}$ is a subgroup for all $H \leq N$ when $N$ is a $\mathcal{F}$-group. Additionally, we have that $|\sqrt[N]{H} : H| < \infty$.

We now define the notion of a determinant of a subgroup of a finite generated nilpotent group. 
\begin{defn}
	Let $N$ be a nilpotent group, and let $H \le N$ be a subgroup. We define the \emph{determinant} $\det(H)$ as $\det(H) = | \sqrt[N]{H}:H|$. If $\map{\varphi}{N_1}{N_2}$ is a morphism of nilpotent groups, we write $\det(\varphi) = \det(\varphi(N_1))$.
\end{defn}

If $\varphi$ is an injective map of torsion-free abelian groups of the same rank, then $\det(\varphi)$ is equal to the usual determinant of the matrix representative of $\varphi$ with respect to a fixed choice of basis. Thus, we have a generalization of the usual notion of determinant to a more general class of groups and group morphisms. If $\map{\varphi}{A}{B}$ is a map of abelian groups where $\varphi(A)$ has rank equal to $B$, then $\det(\varphi) B \subseteq \varphi(A)$.

We need one more invariant for nilpotent groups. 
\begin{defn}
	Let $N$ be a $\mathcal{F}$-group, and let $z \in N$ be a primitive central element of $N$. An \emph{one dimensional central quotient of $N$ associated to $z$} is a quotient $\faktor{N}{H}$ such that $\faktor{N}{H}$ is a $\mathcal{F}$-group where $Z\left( \faktor{N}{H} \right) = \innp{\pi_H(z)}$. 
\end{defn}
For any primitive central element, the existence of an associated one dimensional central quotient of $N$ is guaranteed by \cite[Prop 3.1]{Pengitore_1}; however, uniqueness is not guaranteed motivating the following definition.

\begin{defn}
	Let $N$ be a $\mathcal{F}$-group. We define $\Phi(N)$ to be the smallest integer such that for every primitive $z \in Z(N)$, there exists an one dimensional central quotient $\faktor{N}{H}$ associated to $z$ such that $h \left(\faktor{N}{H}\right) \leq \Phi(N)$. 
\end{defn}
The value $\Phi(N)$ for a $\mathcal{F}$-group is the value found in the statement of \cite[Thm 1.1]{Pengitore_1} and is originally defined in \cite[Defn 3.4]{Pengitore_1}. 

\section{Twisted centralizers and twisted determinants}
\label{sec:twisted}

In this section, we introduce twisted centralizers and study the projections of these subgroups to finite quotients. The key concept for understanding these finite projections is the twisted determinant which we introduce in Definition \ref{def_twisted_det}. From now on, $N$ is a nilpotent group of nilpotency class $c$ and with a fixed automorphism $\map{\varphi}{N}{N}$. 

\begin{defn}
	\label{deftwistedcentralizer}
	Let $N$ be a nilpotent group, and let $\map{\varphi}{N}{N}$ be an automorphism. For every $i \geq 1$, we define the subgroups $$
	N_{i}^\varphi = \{x \: | \: x \: \varphi(x)^{-1} \in \ga_{i}(N) \}
	$$
	and the corresponding subsets
	$$
	X_{i}^\varphi = \{x \: \varphi(x)^{-1} \: | \: x \in N_{i}^\varphi \} \subseteq \ga_{i}(N).
	$$
	We call $N_i^\varphi$ the \emph{$i$-th twisted centralizer} corresponding to the automorphism $\varphi$.
	
\end{defn}
Note that $N_1^\varphi = N$ and $N_{c+1}^\varphi = \{x \in N \suchthat \varphi(x) = x \}$ by definition. If $\varphi = \Inn(y)$ with $y \in N$, then $N_{c+1}^\varphi$ is the centralizer of the element $y$, hence the name. 

The subsets $X_i^\varphi$ play an important role in studying twisted conjugacy classes as they determine if two elements are twisted conjugate when they differ by an element of $\ga_{i}(N)$
\begin{lemma}\label{nearly_twisted_conjugate_lemma}
	Let $N$ be a nilpotent group with an automorphism $\varphi$. For every $x,y \in N$ with $y \in \ga_{i}(N)$, it holds that $$y x \sim_\varphi x \Longleftrightarrow y \in X_i^{\varphi_x}.$$
\end{lemma}
\begin{proof}
Note that $y x \sim_\varphi x$ if and only if there exists $z \in N$ such that $z \: x \: \varphi(z)^{-1} = y \: x$. This last statement is the same as $z \: x \: \varphi(z)^{-1} \: x^{-1} = y$. Equivalently, $y \in X_{i}^{\varphi_x}$.
\end{proof}

The twisted centralizers $N_i^\varphi$ are used to define the maps $\psi_{\varphi,i}$.

\begin{defn}
Let $N$ be a nilpotent group, and let $\varphi \in \Aut(N)$. For each $i$, we define a map $\map{\psi_{\varphi,i}}{N_i^\varphi}{ \faktor{\ga_i(N)}{\ga_{i+1}(N)}}$ as $$\psi_{\varphi,i} (x) = x \: \varphi(x)^{-1} \mod \ga_{i+1}(N).$$
\end{defn}

\begin{lemma}
The map	$\map{\psi_{\varphi,i}}{N_i^\varphi}{\faktor{\ga_i(N)}{\ga_{i+1}(N)}}$ is a group morphism for all $i$.
\end{lemma}
\begin{proof}
	Let $x,y \in N_i^\varphi$. By computation,
	\begin{eqnarray*}
	\psi_{\varphi,i}(xy) &=& xy \: \varphi(xy)^{-1}   \mod \ga_{i + 1}(N) \:  \\
	\text{ } &=& xy \: \varphi(y)^{-1} \: \varphi(x)^{-1}   \mod\ga_{i + 1}(N) \: \\
	\text{ } &=& x \: \psi_{\varphi,i}(y) \: \varphi(x)^{-1}  \mod \ga_{i + 1}(N) \:  \\
	\text{ } &=& \psi_{\varphi,i}(x) \: \psi_{\varphi,i}(y) \: [\psi_{\varphi,i}(y)^{-1}, \varphi(x)]   \mod \ga_{i + 1}(N) \: \\
	\text{ } &=& \psi_{\varphi,i}(x) \: \psi_{\varphi,i}(y).
	\end{eqnarray*}\end{proof}

In order to understand how $N_i^{\varphi}$ relates to $N_{i+1}^\varphi$ for each $i$, we have the following lemma.
\begin{lemma}
With notations as above, we have $\ker(\psi_{\varphi,i}) = N_{i+1}^\varphi$.
\end{lemma}
\begin{proof}
Take any $x \in N_i^\varphi$, then $x \in \ker(\psi_{\varphi,i})$ if and only if $$ x \varphi(x)^{-1} \ga_{i + 1}(N) = \psi_{\varphi,i}(x) = 1 \mod \ga_{i+1}(N)$$ which is equivalent to $x \in N_{i+1}^\varphi$. 
\end{proof}

Each $\psi_{\varphi,i}$ induces an injective map $\map{\bar{\psi}_{\varphi,i}}{\faktor{N_i^{\varphi }}{N_{i+1}^\varphi}}{\faktor{\ga_i(N)}{\ga_{i+1}(N)}}$ of abelian groups.

With the above lemma, we can define a subgroup which will be of importance in the effective upper bound of twisted conjugacy separability.
\begin{defn}
	Let $N$ be a nilpotent group, and let $\varphi \in \Aut(N)$. The set $X_{c}^{\varphi}$ is a subgroup since it is the image of $\psi_{\varphi,c}$. We define the central subgroup $N_{\varphi} = X_{c}^\varphi \le \gamma_{c+1}(N) \le Z(N)$.
\end{defn}

 One can see that $x \in N_\varphi$ are elements of $\ga_{c}(N)$ that take the form $x = y \: \varphi(y)^{-1}$ for some $y \in N$. 

\begin{defn}
	\label{def_twisted_det}
Let $\Ga$ be a $\mathcal{F}$-group, and let $\varphi \in \Aut(N)$. For each $i$, we define $D_{\varphi,i} = \det(\bar{\psi}_{\varphi,i})$. Additionally, we define the \emph{twisted determinant} $D_\varphi$ as 
$$
D_\varphi = \prod_{i=1}^{c} D_{\varphi,i}.$$
\end{defn}

The twisted determinant is the main ingredient for the following proposition.

\begin{prop}\label{lemma_2_blackburn_generalization}
Let $N$ be a $\mathcal{F}$-group and $p$ be a prime. There exists a natural number $k^\ast(p,c)$ such that for every automorphism $\varphi$, natural number $k$, and $x \in N$ with $x \: \varphi(x)^{-1} \in N^{p^{k+ k^\ast(p,c) + v_p(D_\varphi)}}$  there exists $y \in N^{p^k}$ such that $y \: \varphi(y)^{-1} = x \: \varphi(x)^{-1}$. The number $k^\ast(p,c)$ can be chosen such that there exists a constant $C > 0$ with $p^{k^\ast(p,c) + v_p(D_\varphi)} \leq C |D_\varphi|$ for all primes $p$. 
\end{prop}
Observe that $k^\ast(p,c)$ is independent of the choice of automorphism. The above proposition is a twisted generalization of \cite[Lem 2]{Blackburn} and is an effective version of \cite[Thm 4.1.]{Jonas}. 

The proof of Proposition \ref{lemma_2_blackburn_generalization} relies on the following two lemmas.
\begin{lemma}\label{lemma_1}
Let $p$ be a prime, and let $\map{\varphi}{A}{B}$ be a group morphism of abelian groups. If $b \in \varphi(A) \cap \left({p^{k + v_p(\det(\varphi))}} \cdot B \right)$ for some $k \geq 0$, then there exists an $a \in p^k \cdot A$ such that $\varphi(a) = b$.
\end{lemma} 
\begin{proof}
Write $b = {p^{k + v_p(\det(\varphi))}} \cdot g$ for some $g \in B$. Since $b \in \varphi(A)$, we know that $\Ord(g \cdot \varphi(A))$ as an element of the group $\sqrt[B]{\varphi(A)}/\varphi(A)$ divides $p^{k + v_p(\det(\varphi))}$. Since $\det(\varphi) = |\sqrt[B]{\varphi(A)} / \varphi(A)|$, it follows that $g \cdot \varphi(A)$ has order at most $p^{v_p(\det(\varphi))}$. Hence, ${p^{v_p(\det(\varphi))}} \cdot g \in \varphi(A)$. Let $\tilde{a} \in A$ with $\varphi(\tilde{a}) = g$, then the element $a =  p^k \: \tilde{a}$ is our desired element.
\end{proof}

We reproduce the proof of \cite[Lem 2]{Blackburn} in order to estimate the associated value that is constructed.
\begin{lemma}\label{blackburn_lemma}
	Let $N$ be a $\mathcal{F}$-group of nilpotency class $c$, and let $p$ be a prime. There exists an integer $k(p,c) \geq 0$ such that if $x \in N^{p^{k+k(p,c)}}$, then there exists $y \in N$ such that $x = y^{p^k}$. Additionally, $p(k,c)$ can be chosen such that $p^{k(p,c)}\leq c! \:$ for all primes $p$.
\end{lemma}
\begin{proof}
For a natural number $m$, we let $d_p(m)$ be the largest integer such that $p^{d_p(m)} \leq m$. We show that the lemma holds for the value $k(p,c) = \sum_{i=1}^{c} d_p(i)$. We proceed by induction on nilpotency class length, and since the statement is evident for abelian groups, we may assume that $N$ has nilpotency class $c > 1$. For $x \in N^{p^{k + k(p,c)}}$, we write $x  = \prod_{i=1}^r x_i^{p^{k+k(p,c)}}$ where $x_i \in N$ for all $i$. Then \cite[Thm 6.3]{Hall_notes} implies that we may write
	$$
	x = \prod_{i=1}^{r}x_{i}^{p^{k+k(p,c)}} = \pr{\prod_{i=1}^r x_i}^{p^{k+k(p,c)}}  \prod_{i=2}^{c}y_i^{{p^{k+k(c,p)} \choose i}} = \prod_{i=1}^{c}y_i^{{p^{k+k(c,p)} \choose i}}
	$$
	where $y_i \in \ga_i(N)$ for each $i$ and $y_1 = \prod_{i=1}^r x_i$. We may write the binomial coefficient
	$$
	{p^{k + k(p,c)} \choose i} = {p^{k + k(p,c)} \choose p^{\ell_i} \: u_i}
	$$
	where $\GCD(u_i,p) = 1$. Thus, this value is divisible by $p^{k+k(p,c) - \ell_i}$; therefore, we write
	$$
	{p^{k + k(p,c)} \choose i} = m_i \: p^{k + k(p,c) - \ell_i}.
	$$
	Since $N$ has nilpotent class $c$ and $y_i \in \ga_i(N)$, we have $g_{p^{\ell_i} \: u_i}=1$ for $p^{\ell_i} \: u_i > c$. If $p^{\ell_i}u_i \leq c$, then $\ell_i \leq d_p(c)$. Letting $s_i = d_p(c) - \ell_i$, we write
	$$
	{p^{k + k(p,c)} \choose i} = m_i \: p^{k + k(p,c) - \ell_i} = m_i p^{k + k(p,c) + s_i - d_p(c) }
	$$
	for $i \leq c$. Subsequently, for $i \leq c$ we have
	$$
	y_i^{{p^{k + k(p,c)} \choose i} } = y_i^{m_i p^{k + k(p,c) + s_i - d_p(c) }} = (y_i^{m_i \: p^{s_i}})^{p^{k + k(p,c) - d_p(c)}}.
	$$
	Letting $z_i = y_i^{m_i \: p^{s_i}}$, we may write
	$$
	x = 
	\prod_{i=1}^{r}x_{i}^{p^{k+k(p,c)}} = \prod_{i=1}^{c} z_i^{p^{k+k(p,c) - d_p(c)}} = \prod_{i=1}^c z_i^{p^{k + k(p,c-1)}}.
	$$
	where $z_i \in [N,N]$ for $2 \leq i \leq c$. Then \cite[Lem 1.3]{Hall_notes} implies that $\innp{z_1,[N,N]}$ has nilpotency class strictly less than $c$. Since $\set{z_1,z_2, \cdots, z_{c}} \in \innp{z_1,[N,N]}$, the inductive hypothesis gives us a $y \in \innp{z_1,[N,N]}$ such that $x = y^{p^k}$. By construction, $k(p,c) = \sum_{i=1}^{c}d_p(i)$ which implies $p^{k(p,c)} \leq \prod_{i=1}^{c}p^{d_p(i)}\leq c!.$
\end{proof}
\begin{proof}[Proof of Proposition \ref{lemma_2_blackburn_generalization}]
Let $k(p,c)$ be the natural number from Lemma \ref{blackburn_lemma}. We proceed by induction on $i$ where $x \in N_i^\varphi$. If $i = c + 1$, or equivalently, $x \: \varphi(x)^{-1} \in \ga_{c+1}(N) =1$, then we take $y = 1$.

Now assume that $x \: \varphi(x)^{-1} \in \ga_i(N) \cap N^{p^{k}}$ for $k$ sufficiently big. Letting $A = \faktor{N_i^\varphi}{N_{i+1}^\varphi}$ and $B = \faktor{\ga_i(N)}{\ga_{i+1}(N)}$, we consider the injective morphism of abelian groups $\map{\bar{\psi}_{\varphi,i}}{A}{B}$ induced by $\psi_{\varphi,i}$. Lemma \ref{blackburn_lemma} gives us a $z \in N$ such that $x \: \varphi(x)^{-1} = z^{p^{k - k(p,c)}}$ which implies $\bar{\psi}_{\varphi,i}(x) \in B^{p^{k - k(p,c)}}$. Lemma \ref{lemma_1} shows that there exists $x_i \in (N_i^\varphi)^{p^{k - k(p,c) - v_p(\det(\bar{\varphi}))}}$ such that $\bar{\psi}_{\varphi,i}(x \: N_{i+1}^\varphi)
 = \bar{\varphi_i}(x_i \: N_{i+1}^\varphi)$. Hence, $x_i^{-1} \: x \in N_{i+1}^\varphi$, and since $x_i,x \in N^{p^{k - k(p,c) - v_p(D_i)}}$, it follows that $\tilde{x} \in N^{p^{k - k(p,c) - v_p(D_i)}}$. Thus, $\tilde{x}\: \varphi(\tilde{x})^{-1} \in \ga_{i+1}(N) \cap N^{p^{k - k(p,c) - v_p(D_i)}}.$ Induction implies the theorem follows with the constant given by $k^\ast(p,c) = (c-1) \: k(p,c)$.
 
 We finish by noting Lemma \ref{lemma_2_blackburn_generalization} implies that $k(p,c) \leq c\,!$, and thus, $$p^{k^\ast(p,c) + v_p(D_\varphi)} = p^{(c-1)\cdot k(p,c) + v_p(D_\varphi)} \leq (p^{k(p,c)})^{c-1} p^{v_p(D_\varphi)} \leq (c)^{c - 1} |D_\varphi|.$$ \end{proof}

We denote by $\bar{\varphi}_{p^k}$ the automorphism induced by $\varphi$ on the quotient $\faktor{N}{N^{p^k}}$.

\begin{cor}\label{twisted_central_pullback_matrix_reduction}
Let $N$ be a $\mathcal{F}$-group of nilpotency class $c$, and let $\map{\varphi}{N}{N}$ be an automorphism. Let $p$ be a prime and $k$ be some natural number. Take $k^\ast(p,c-1)$ as in Proposition \ref{lemma_2_blackburn_generalization}, and define $k_0 = v_p(D_\varphi) + k^\ast(p,c-1)$. Define $\map{\rho}{N / N^{p^{k+k_0}}}{N / N^{p^k}}$ as the natural projection such that $\pi_{p^k} = \rho \circ \pi_{p^{k+k_0}}$. Then $\rho \left( \left( \faktor{N}{N^{p^{k + k_0}}} \right)_{\bar{\varphi}_{p^{k+k_0}}}\right) = \pi_{p^k}(N_\varphi)$.
\end{cor}
\begin{proof}
First assume that $x \in N_\varphi$, or equivalently, $x = y \: \varphi(y)^{-1} \in \ga_{c}(N)$ for some $y \in N$. Note that $\pi_{p^{k+k_0}}(x) \in \gamma_c(N / N^{p^{k+k_0}})$ since $\pi_{p^{k+k_0}}$ is a group morphism. Now $$\pi_{p^{k+k_0}}(x) =\pi_{p^{k+k_0}} (y \varphi(y)^{-1})= \pi_{p^{k+k_0}}(y) \: \varphi_{p^{k+k_0}}(\pi_{p^{k+k_0}}(y))^{-1} \in \left( \faktor{N}{N^{p^{k + k_0}}} \right)_{\bar{\varphi}_{p^{k+k_0}}};$$ hence,  $\rho(\pi_{p^{k+k_0}}(x)) = \pi_{p^k}(x) \in \rho \left( \left( \faktor{N}{N^{p^{k + k_0}}} \right)_{\bar{\varphi}_{p^{k+k_0}}}\right)$. Hence, $\pi_{p^k}(N_\varphi) \subseteq \rho \left( \left( \faktor{N}{N^{p^{k + k_0}}} \right)_{\bar{\varphi}_{p^{k+k_0}}}\right)$.

For the other inclusion, let $x \in \ga_c(N)$ with $\pi_{p^{k}}(x) \in \rho \left( \left( \faktor{N}{N^{p^{k + k_0}}} \right)_{\bar{\varphi}_{p^{k+k_0}}}\right)$. Thus, there exists $y \in N$ such that $y \: \varphi(y)^{-1} \in x \: N^{p^{k+k_0}}$. Applying Proposition \ref{lemma_2_blackburn_generalization} to $N / \ga_c(N)$ and the map induced by $\varphi$, there exists $z \in N^{p^k}$ such that $z \: \varphi(z)^{-1} \in y \: \varphi(y)^{-1} \ga_{c}(N)$. Hence, $y^{-1} \: z \: \varphi(y^{-1} \: z)^{-1} \in N_\varphi$. Since $\pi_{p^k}(z) = \pi_{p^k}(\varphi(z)) = 1$, we get
$$
\pi_{p^k}(x) = \pi_{p^k}(y \: \varphi(y)^{-1}) = \pi_{p^k}(z^{-1} \: y \: \varphi(z^{-1} \: y)^{-1}) \in \pi_{p^k}(N_\varphi).
$$
We conclude that $\rho \left( \left( \faktor{N}{N^{p^{k + k_0}}} \right)_{\bar{\varphi}_{p^{k+k_0}}}\right) = \pi_{p^k}(N_\varphi)$.
\end{proof}

\section{Bounding the Twisted Determinant}

In the previous section, we identified the twisted determinant as the crucial factor for studying the twisted centralizers in finite quotients. The goal of this section is to bound the twisted determinant of an automorphism in terms of the norm of that automorphism. 

Before we start, we provide some propositions and examples that relate the determinates of subgroups with the norms of those subgroups.
\begin{ex}
	\label{firstlength}
	Consider the group $G = \Z$ with standard generating subset $S = \{ \pm 1\}$. The norm of the subgroup $n \Z$ with $n \in \N$ satisfies $$\Vert n \Z \Vert_S = n = [\Z:n\Z] = \det(n\Z).$$\end{ex}

We generalize the previous example to all finitely generated abelian groups.

\begin{prop}\label{det_bound_length_subgroup}
Let $A$ be a finitely generated abelian group with a finitely generating subset $S$. There exists a constant $C > 0$ such that if $H \leq A$ is a subgroup, then $\det(H) \leq C \: (\|H\|_S)^{\text{rank}_\Z (H)}$.
\end{prop}

\begin{proof}
Since $T(A)$ is always a finite normal subgroup with $T(A) \le \sqrt[A]{H}$, we may assume that $A = \Z^k$. Moreover, we may assume that $S$ is the standard generating subset since the statement is invariant under changing the generating subset.

Let $\ell$ be the rank of $H$. There exist $h_1, \ldots, h_l \in H$ such that $H' = \innp{h_1, \cdots, h_\ell}$ has rank $l$ and $\|h_i\|_S \leq \|H\|_S$. Thus, we may restrict our attention to subgroups of the form $H'$ since $\det(H) \leq \det(H')$ and $\Vert H^\prime \Vert_S \leq \Vert H \Vert_S$. 

Let $\map{\pi}{\Z^k}{\Z^\ell}$ be the projection to the first $\ell$ coordinates. By relabeling coordinates as necessary, we may assume that $\pi|_H$ is injective. Since $\pi|_H$ is injective on $H$, it follows that $\pi|_{\sqrt[A]{H}}$ is injective. Thus, $[\pi(\sqrt[A]{H}) : \pi(H)] = [\sqrt[A]{H} : H ]$, and since $\pi(\sqrt[A]{H}) \leq \sqrt[\Z^\ell]{\pi(H)} = \Z^\ell$, it follows that $\det(\pi(H)) \geq \det(H)$. Thus, without loss of generality, we may assume that $H$ has full rank in $A$. 

Let $X = [ h_1, h_2, \hdots, h_\ell]^T$. Since $H$ is full rank, it follows that $X$ has non-zero determinant. Observing that $\det(X) = [\Z^k:H] = \det(H)$ and that each coefficient of $X$ is bounded by $\|H\|_S$, the formula for the determinant gives the desired bound.
\end{proof}

Applying the above proposition to group morphisms, we have the following.
\begin{cor}\label{abelian_image_bound}
Let $A$ be a finitely generated abelian group with a finite generating subset $S$. There exists a constant $C>0$ such that for every morphism $\map{\varphi}{B}{A}$ of abelian groups, it follows that $$\det(\varphi) \leq C (\| \varphi(B) \|_S)^{\text{rank}(\varphi(B))} \leq C (\Vert \varphi \Vert_S)^{\text{rank}(\varphi(B))}.$$
\end{cor}
\begin{proof}
Our result follows from Proposition \ref{det_bound_length_subgroup} since the value only depends on the image.
\end{proof}

The following example shows that one cannot bound the norm of $\varphi(B)$ by the determinant of $\varphi$.
\begin{ex}
	Consider the group $\Z^2$ with the standard generating subset $S$ and subgroup $H = \langle \left( 1, n \right) \rangle$. A computation shows that $$\Vert H \Vert_S = n + 1,$$ but $\det(H) = 1$.
\end{ex}

The following lemma will be useful to estimate the norm of the kernel of group morphisms.
\begin{lemma}\label{unique_preimage_bound}
Let $S$ be a generating subset for $\Z^k$. There exists a constant $C > 0$ such that for every injective group morphism $\map{\varphi}{\Z^k}{\Z^k}$ and every element $x \in \Z^k$, it holds that $$\|x\|_S \leq C \frac{ \: (\max\{ \|\varphi\|_S, \|\varphi(x)\|_S \})^k}{\vert \det(\varphi) \vert}.$$
\end{lemma}\begin{proof}
	The statement is invariant under changing the generating subset, so we assume that $S$ is the standard generating subset. Writing $x = (x_i)_{i=1}^k$, Cramers rule implies
$$
x_i = \frac{\det(\varphi_i)}{\det  (\varphi)}
$$
where $\varphi_i$ is the morphism given by replacing the $i$-th column of the matrix representative of $\varphi$ by $\varphi(x)$. Moreover, each entry of the matrix representative of $\varphi_i$ is bounded by $\max\{\|\varphi\|_S,\| \varphi(x) \|_S\}.$ Thus, the explicit formula for the determinant of a $(k \times k)-$matrix gives our result.
\end{proof}

\begin{lemma}\label{ker_morphism_inequalityab}
Let $A$ be a finitely generated abelian group of rank $k > 0$ with a finite generating subset $S'$ and $m$ a fixed natural number. There exists a constant $C>0$ such that for every group morphism $\varphi: \Z^\ell \to A$ and generating subset $S$ with $\vert S \vert \leq m$ for $\Z^\ell$ it holds that $$\|\ker(\varphi)\|_S \leq C (\|\varphi\|_{S,S'})^{k}.$$
Moreover, there exists $m^\prime$, not depending on $\varphi$, such that this bound can be achieved with a generating subset with $\leq m^\prime$ elements.
 \end{lemma}
\begin{proof}
Without loss of generality, we may assume that $A$ is torsion-free, so $A = \Z^k$. The statement is invariant under changing the generating subset of $A$, so we may assume that $S^\prime$ is the standard generating subset. Additionally, we assume that $\varphi(\Z^\ell)$ has rank $k$ by taking a quotient of $A$ if necessary. 

First, we prove the lemma for the standard generating subset $S = \{e_1, \ldots, e_l \}$ for $\Z^l$. We assume that $\varphi(e_i)$ are linearly independent for $1 \leq i \leq k$, and let $D = \det(\varphi(e_1),\varphi(e_2),\hdots,\varphi(e_k))$.
Note that there exists a constant $C_0$ such that $\vert D \vert \leq C_0 \Vert \varphi \Vert_{S,S^\prime}^k$.

For each $1 \leq j \leq l - k$, we construct a vector $v_j \in \ker (\varphi)$. Apply Lemma \ref{unique_preimage_bound} gives a $x_j \in \Z^k$ such that $\varphi(x_j) = D \: \varphi(e_{k + j})$ and $\vert x_j \vert_S \leq C_1 (\|\varphi\|_{S,S'})^{k}$. Letting $\map{\iota}{\Z^k}{\Z^l}$ be the natural inclusion, the vector $v_j = \iota(x_j) - D \: e_{k + j}$ is in $\ker(\varphi)$ by construction. The vectors $v_j$ generate a finite index subgroup of $\ker(\varphi)$. Since $S$ is assumed to be the standard generating subset, an easy computation shows that $\|\ker(\varphi)\|_S \leq m \: \max\{\|v_j \| \}$. Therefore, $\|\ker(\varphi)\|_S \leq C (\|\varphi\|_{S,S'})^{k}$ and the bound $m^\prime$ also follows from this construction.

Now assume that $S = \{s_1, \ldots, s_m\}$ is an arbitrary generating subset for $\Z^\ell$. Consider the morphism $f: \Z^m \to \Z^\ell$ which maps $e_i$ to the generator $s_i$. For the composition $\varphi \circ f$, we have $$\|\ker(\varphi \circ f )\|_S \leq C (\|\varphi\|_{S,S'})^{k}.$$ Now $f$ maps the generators of this kernel to generators of the kernel of $\varphi$, and moreover, $\Vert f(x) \Vert_{S} \leq \Vert x \Vert$ for every $x \in \Z^m$. Thus, the lemma follows. 
\end{proof}

\begin{ex}
	Fix $n \in \N$ and take $n$ distinct primes $p_1 < \ldots < p_n$. Consider the matrix  given by $$A_j = \begin{pmatrix} p^j_1 & 0 & \dots & 0 & 1 \\0 & p^j_2 & \dots & 0 & 1 \\ \vdots & \vdots & \ddots & \vdots & \vdots \\ 0 & 0 & \dots & p^j_n & 1\end{pmatrix}$$ and the corresponding linear map $\varphi_j: \Z^{n+1} \to \Z^{n}$. Let $S$ and $S^\prime$ be the standard generating sets for $\Z^{n+1}$ and $\Z^n$ respectively. If we write $k = p_1 \dots p_n$, then the kernel of $\varphi_j$ is generated by the element $\left( \frac{k^j}{p^j_1}, \ldots, \frac{k^j}{p^j_n}, -k^j\right)$. A computation shows that $$\Vert \ker(\varphi_j) \Vert_S = \Vert \left( \frac{k^j}{p^j_1}, \ldots, \frac{k^j}{p^j_n}, -k^j\right) \Vert_S \geq k^j \geq p_n^{jn}.$$ On the other hand, $\Vert \varphi_j \Vert_{S^\prime} = \max \{ p_n^j, n \}$. This example implies that the bound in Proposition \ref{ker_morphism_inequalityab} is optimal.
\end{ex}

\begin{prop}
	\label{rootcommutator}
For every $m \geq 0$, there exists a constant $C$ such that for every nilpotent group $N$ and any generating subset $S$ with $\vert S \vert \leq m$, it holds that 
$$\Vert \sqrt[N]{ \gamma_2(N) } \Vert_S \leq C$$
\end{prop}

\begin{proof}
If $s_1, \ldots, s_k$ are the generators of $N$, then consider the group morphism $\varphi: \Z^k \to \faktor{N}{ \sqrt[N]{ \gamma_2(N) }}$ with $\varphi(e_i) = s_i  \sqrt[N]{ \gamma_2(N) }$. From Lemma \ref{ker_morphism_inequalityab}, it follows that there exists a $C$ such that $\Vert \ker(\varphi) \Vert \leq C$. For every generator $x_j \in \Z^k$ for $\ker(\varphi)$, fix an element $y_j \in N$ such that $\Vert y_j \Vert_S \leq \Vert x_j \Vert$ and $\varphi(x_j) = y_j \sqrt[N]{ \gamma_2(N) }$. The group $\sqrt[N]{ \gamma_2(N) }$ is now generated by the $y_j$ and $\gamma_2(N)$, where $\Vert \gamma_2(N) \Vert_S \leq 4.$  
\end{proof}
	
\begin{prop}\label{ker_morphism_inequality}
	Let $A$ be a finitely generated abelian group of rank $k > 0$ with a finite generating subset $S'$. For every $m > 0$, there exists a constant $C>0$ such that if $N$ is a $\mathcal{F}$-group with a finite generating subset $S$ with $\vert S \vert \leq m$ and $\map{\varphi}{N}{A}$ is a group morphism, then $$\|\ker(\varphi)\|_S \leq C (\|\varphi\|_{S,S'})^{\text{rank}_\Z(A)}.$$
	Moreover, there exists $m^\prime$ such that this bound can be achieved with a generating subset with $\leq m^\prime$ elements.
\end{prop}

\begin{proof}
We can assume that $A$ is torsion-free. First note that $\sqrt[N]{\gamma_2(N)}$ is a subgroup of $\ker{\varphi}$ whose norm is bounded as described in Proposition \ref{rootcommutator}. So it suffices to bound the norm of the kernel of $\bar{\varphi}: \faktor{N}{\sqrt[N]{\gamma_2(N)}} \to A$, and thus, Lemma \ref{ker_morphism_inequalityab} gives our proposition. The statement about the number of generators also follows from the same lemma.
\end{proof}

\begin{thm}\label{twisted_pullback_norm_bound}
Let $N$ be a $\mathcal{F}$-group with a finite generating subset $S$, and let $i>0$. Then there exists a natural number $k$ and constant $C > 0$ such that  $\|N_i^\varphi\|_S \leq C \: (\|\varphi\|_S)^{k}$ for any $\varphi \in \Aut(N)$.
\end{thm}
\begin{proof}
We proceed by induction on $i$ that every $N_i^{\varphi}$ satisfies the condition of the theorem, with the additional assumption that the number of elements in such a generating subset is uniformly bounded by a constant $m_i$. For $i=1$, we have that $\ga_1(N) = N$, and thus, $N_1^\varphi  = N$. Our theorem is now evident for this case with $m_1 = \vert S \vert$ and $k = 1$.

Now assume that the result holds for the group $N_i^\varphi$ with the constant $C_1$, integer $k_1$ and number of generators $m_i$. By assumption, there exists a finite generating subset $S'$ for $N_{i}^\varphi$ with $|s|_S \leq C_1 (\|\varphi\|_S)^{k_1}$ for all $s \in S'$ and $\vert S^\prime \vert \leq m_i$.  

For the quotient $\faktor{\ga_i(N)}{\ga_{i+1}(N)}$, we fix the generating subset $S_i$ given by the projections of commutators of length $i$. Denote $r_i$ to be the rank of $\faktor{\ga_{i}(N)}{\ga_{i+1}(N)}$ and fix the constant $C_2$ as in Proposition \ref{ker_morphism_inequality} for the nilpotency class $c$ with the generating subset $S_i$, which is independent of the group morphism $\varphi$.

From the computation after Definition \ref{def_normauto}, it follows that $\|\varphi\|_{S',S_i} \leq C_1 \pr{\|\varphi\|_S}^{k_1} \: \| \varphi \|_{S,S_i}.$ Take $C_3$ as any constant such that $\|s\|_S \leq C_3$ for all $s \in S_i$. (It is easy to give an explicit form for the constant $C_3$.) Hence, it follows that $\|\varphi\|_{S,S_i} \leq C_3 \: \|\varphi\|_{S}$, and consequently, $\|\varphi\|_{S',S_i}  \leq C_1 \: C_3 \: \pr{\|\varphi\|_S}^{k_1 + 1}$.

From the computation after Definition \ref{def_normsubgroup}, we have $\|N_{i+1}^\varphi\|_S \leq C_1 \: \pr{\|\varphi\|_S}^{k_1} \: \|N_{i+1}^\varphi\|_{S'}$. By using Proposition \ref{ker_morphism_inequality}, we get that
\begin{align*}
\|N_{i+1}^\varphi\|_S &\leq C_1 \: \pr{\|\varphi\|_S}^{k_1} \: \|N_{i+1}^\varphi\|_{S'} \\
&\leq C_1 \: C_2 \: \pr{\|\varphi\|_S}^{k_1} \pr{\|\varphi\|_{S',S_i}}^{r_i} \\
&\leq C_1^{r_i + 1} \: C_2 \: C_3^{r_i} \pr{\|\varphi\|_S}^{k_1 + k_1 \: r_i + r_i}.
\end{align*}
It follows that there exists a bound $m_{i+1}$ on the number of generators by Proposition \ref{ker_morphism_inequality}.
\end{proof}

An important application of Theorem \ref{twisted_pullback_norm_bound} is to bound the twisted determinant of an automorphism in terms of its norm.
\begin{cor}\label{twisted_determinant_bound}
	Let $N$ be a $\mathcal{F}$-group with generating subset $S$. There exists a constant $C > 0$ and $k \in \N$ such that for every automorphism $\map{\varphi}{N}{N}$ the twisted determinant $D_\varphi$ satisfies
	$$
 \vert 	D_\varphi \vert \leq C \pr{\|\varphi\|_S}^k.
	$$
\end{cor}
\begin{proof}
	It suffices to give such a bound for every determinant $D_i$ of $\psi_i$. For this bound, we use Theorem \ref{twisted_pullback_norm_bound} on the group $N_i^\varphi$ to find a generating subset $S^\prime$ whose word length $\Vert S^\prime \Vert_S \leq C (\Vert \varphi \Vert_S)^k$ for some $C, k \in \N.$ In particular, the norm $\|\psi_i(N_i^\varphi)\|_S \leq C (\Vert \varphi \Vert_S)^{k+1}$. We next apply Corollary \ref{abelian_image_bound} to the group morphism $\psi_i$ to find a bound on the determinant $D_i$.
\end{proof}

As a final application of the above bounds, we have the following estimate which is essential for Theorem \ref{main_thm}.
\begin{cor}\label{lemma_2_blackburn_generalization_bound}
	Let $N$ be a $\mathcal{F}$-group with a finite generating subset $S$. Let $p$ be prime, and let $k^\ast(p,c)$ be the constant from Corollary \ref{twisted_central_pullback_matrix_reduction}. Then there exists some constant $C >0$ and an integer $k$ such that $p^{k^\ast(p,c) + v_p(D_\varphi)} \leq C \pr{\|\varphi\|_S}^{k}$ for every automorphism $\varphi: N \to N$.
\end{cor} 
\begin{proof}
This follows immediately from the bounds in Proposition \ref{lemma_2_blackburn_generalization} and Corollary \ref{twisted_determinant_bound}.
\end{proof}

\section{Precise Conjugacy Separability of Two Step Nilpotent Groups}

Before we give a proof of the main theorem, we apply the techniques that we have developed so far in a specific setting, namely conjugacy separability of nilpotent groups of nilpotency class $2$. The main goal of this section is to give a precise computation of the asymptotic behavior of $\Conj_{N,S}(n)$ where $N$ is a $\mathcal{F}$-group (see Theorem \ref{two_step_conjugacy_precise}). We start with some preliminary results and observations. 

Let $N$ be a $\mathcal{F}$-group of nilpotency class $2$ with $x \in N$. From Definition \ref{deftwistedcentralizer}, we have  $$N_{2}^{\Inn(x)} =  \{y \in N | \: y \: x \: y^{-1} \: x^{-1} \in \ga_2(N) \} = N$$ since $N$ is a  nilpotent group of nilpotency class $2$. We also observe that the map $$\psi_{\Inn(x),2}: N \to \ga_2(N)$$ is given by $\psi_{\Inn(x),2}(y) = [y,x]$. Note that for every $x, y, z \in N$, it holds that $$[z,x y] = z x y z^{-1} y^{-1} x^{-1} = z x z^{-1} x^{-1} x z y z^{-1} y^{-1} x^{-1} = [z,x] [z,y].$$ This calculation can also be observed from the fact that $N_3^{\Inn(x)} = 1$ for all $x \in N$.

In the case of inner automorphisms, we have a stronger version of Theorem \ref{twisted_pullback_norm_bound}.
\begin{prop}\label{important_morphism_bound}
	Let $N$ be a $\mathcal{F}$-group with a finite generating subset $S$ such that $c(N) = 2$. Then there exists a $C > O$ such that $\|\im(\psi_{\Inn(x),2})\|_{S} \leq C \: \sqrt{\Vert x \Vert_S}$ for every $x \in N$.
\end{prop}
\begin{proof}
	Write $S^\prime = \{[s_i,s_j] \: | \: s_i,s_j \in S\}$ as a finite generating subset for $\ga_2(N)$. From \cite{Osin}, it follows that there exists a constant $C > 0$ such that $$\Vert x \Vert_{S} \leq C \sqrt{ \Vert x \Vert_{S^\prime}}$$ for all $x \in \gamma_2(N)$. 
The subgroup $\im(\psi_{\Inn(x),2})$ is generated by elements of the form $[s,x]$ with $s \in S$. It follows immediately that $\Vert [s,x] \Vert_{S^\prime} \leq \Vert x \Vert_S$, and hence, $\Vert \im(\psi_{\Inn(x),2}) \Vert_{S} \leq C \sqrt{\Vert x \Vert_S}$.
	\end{proof} 
	
Another crucial ingredient for the main results is separability of central subgroups. We start with an easy example which we will use in the proof of the next proposition.

\begin{ex}
	\label{cyclicsubgroupseperability}
Take the group $\Z$ with standard generating subset $S = \{\pm 1\}$. Consider any non-trivial subgroup $H = \langle h \rangle \neq 0$. We can seperate any element $x \in \Z, x \notin H$ in the finite quotient $\faktor{\Z}{H}$. The order of this finite quotient is $\vert h \vert = \Vert H \Vert_S$. Note that $\faktor{\Z}{H}$ is a direct sum of subgroups with order a prime power. In particular we can separate every element $x \notin H$ from $H$ in a quotient of the form $\faktor{\Z}{p^k \Z}$ with $p^k \leq \Vert H \Vert_S$. 

For a general generating subset $S^\prime$ for $\Z$, we find that there exists a $C > 0$ such that every element can be separated in a finite quotient $\faktor{\Z}{p^k \Z}$ with $p^k \leq C \Vert H \Vert_{S^\prime}$. 
\end{ex}

The following result generalizes this example to nilpotent groups.
	
\begin{prop}
	\label{centralsubgroup}
	Let $N$ be a $\mathcal{F}$-group of nilpotency class $c$ with finite generating subset $S$. There exists a constant $C$ such that for all central subgroups $H \le Z(N)$ and every $x \in N, x \notin H$, we can separate $x$ and $H$ in a finite quotient $Q$ with $$\vert Q \vert \leq C \max\{ \Vert H \Vert_S, \log \left( \Vert x \Vert_S \right) \}^{c \hspace{1mm} \Phi(N)}.$$ 
\end{prop}

\begin{proof}
	Assume that $H$ is a central subgroup and $x \in N \setminus H$. We construct a finite quotient which separates $x$ from $H$ in two different cases, according to whether or not $x \in H_0 = \sqrt[N]{H}$.
	
First assume that $x \in H_0$. By taking a quotient $M$ of $N$ if necessary, we can assume that $Z(N)$ has rank $1$ and that this quotient satisfies $h(M) \leq \Phi(N)$. Let $z$ be a generator of $Z(M)$. From Example \ref{cyclicsubgroupseperability}, it follows that there exists a prime power $p^k$ such that $x$ and $H$ are separated in $\faktor{Z(M)}{p^k Z(M)}$ with $p^k \leq C \Vert H \Vert_{S^\prime}$ for $S^\prime$ a generating subset for $Z(M)$. There exists a constant $C^\prime$ such that $\Vert H \Vert_{S^\prime}\leq C^\prime \Vert H \Vert_S^{c}$ by \cite{Osin}. Now take $k(p,c)$ as in Lemma \ref{blackburn_lemma}, and consider the quotient $\faktor{M}{M^{p^{k+k(p,c)}}}$. Note that $M^{p^{k+k(p,c)}} \cap Z(M) \le p^k \cdot Z(M)$ by Lemma \ref{blackburn_lemma}, and thus, $x$ and $H$ are seperated in this finite quotient. The order of this quotient is bounded by $ p^{h(M) k(p,c)} p^{h(M)k} \leq \left(c!\right)^{\Phi(N)} \hspace{1mm} C \Vert H \Vert^{c \hspace{0.5mm} \Phi(N)}$. 

Now assume that $x \notin H_0$. In this case, consider the quotient group $\faktor{N}{H_0}$ which is a torsion-free nilpotent group with $\pi_{H_0} (x) \neq 1$. From \cite[Thm 1.1]{Pengitore_1}, where $\Psi_{\text{RF}}(N) = \Phi(N)$, it follows there exists a constant $C$ such that $x$ is separated from $H$ in a quotient of order $\leq C \log \left( \Vert x \Vert_S \right)^{c \hspace{0.5mm} \Phi(N)}$. Since we have a bound in both cases, the proof is finished.\end{proof}

Now that we have all of the necessary tools, we can give an upper bound for the conjugacy separability of $\mathcal{F}$-groups of nilpotency class $2$. The upper bound matches the lower bound given in \cite{Pengitore_1}, but as mentioned before there is a gap in the proof of this lower bound and hence the asymptotic behavior is not yet fully understood.
\begin{thm}\label{two_step_conjugacy_precise}
	Let $N$ be a $\mathcal{F}$-group with a finite generating subset $S$ such that $c(N) = 2$. Then $\Conj_{N,S}(n) \preceq n^{\Phi(N)}$.
\end{thm}
\begin{proof}
Suppose $x, y \in N$ such that $\|x\|_S,\|y \|_S \leq n$ and $x \nsim_{\text{id}} y$. If $\pi_{\ab}(x) \nsim_{\text{id}} \pi_{\ab}(y)$, or equivalently, if $\pi_{\ab}(x) \neq \pi_{\ab}(y)$, then \cite[Thm 1.1.]{Pengitore_1} implies that there exists a surjective group morphism $\map{\pi}{N / [N,N]}{Q}$ such that $\pi(\pi_{\ab}(x \: y^{-1})) \neq 1$ and where $|Q| \leq C \: \log( \|x \: y^{-1}\|)$. Since $\pi(\pi_{\ab}(x))$ and $\pi(\pi_{\ab}(y))$ are non-equal central elements, they are not conjugate. Thus, $D_{N}([x]_{\id},y) \leq 2 \: C \log(2 n).$
	
	Now we may assume that $y = x \: z$ where $\|z\|_S \leq n$. Consider the group morphism $\psi_{\Inn(x),2}$ as before. Proposition \ref{important_morphism_bound} implies $\|\im(\psi_{\Inn(x),2})\|_S \leq C_1 \: \sqrt{n}$ for some $C_1 \in \N$. Proposition \ref{centralsubgroup} implies there exists a surjective group morphism $\pi: N \to Q$ such that $\pi(z) \notin \pi(\im(\psi_{\Inn(x),2}))$ and where $|Q| \leq C_2 \: (\Vert \im(\psi_{\Inn(x),2})\Vert_S)  ^{2 \: \Phi(N)} \leq C_1^{2\Phi(N)} C_2 \: n^{\Phi(N) }$. We claim that $\pi(y) \nsim_{\text{id}} \pi(x \: z)$. Suppose for a contradiction there exists an $a \in N$ such that $\pi(a \: x \: a^{-1}) = \pi(x \: z)$. That implies $\pi(a \: x \: a^{-1} \: x^{-1}) = \pi(z)$. Thus, $\pi(z) \in \pi(\im(\psi_{\Inn(x),2}))$ which is a contradiction. Hence, $D_N([x]_{\id},x \: z) \leq  C_1^{2\Phi(N)} C_2 \: n^{\Phi(N)}$, and subsequently, $\Conj_{N,S}(n) \preceq n^{\Phi(N)}$.
\end{proof}

\section{Effective Twisted Conjugacy Separability}

In this section we prove the main results of this paper. We start with the abelian case.

\begin{prop}\label{twisted_abelian_separability}
	Let $A$ be a finitely generated abelian group with a finite generating subset $S$. For every $x \in A$, it holds that  $$\Conj_{A,x,S}^{\varphi}(n) \preceq \|\varphi\|_S \: \log \left(\Vert x \Vert_S + n \right)$$ for all $\varphi \in \Aut(A)$. Subsequently, $\Conj_{A,S}^{\varphi}(n) \preceq \|\varphi\|_S \: \log(n)$ and $\Tconj_{A,S}(n) \preceq n \log(n)$.
\end{prop}

\begin{proof}
	Let $\varphi \in \Aut(A)$, and take $x \in A$. Suppose $y \in A$ satisfies $\|y\|_S \leq n$ and $x \nsim_\varphi y$. We may write
	$$
	[x]_\varphi = \{z + x - \varphi(z) \: | \: z \in A \} = \{ x + \pr{ I_A - \varphi}(z) \: | \: z \in A \}
	$$
	where $I_A: A \to A$ is the identity map. 
	It thus holds that $y \notin [x]_\varphi$ if and only if $y-x \notin \im(I_A - \varphi)$. Observe that $\im(\psi_{\varphi,1}) = \im(I_A - \varphi)$ if we use the notations from Section \ref{sec:twisted}. For $s \in S$, we observe that $\{(I_A - \varphi)(s) \: | \: s \in S \}$ is a generating subset for $\im(I_A - \varphi)$. Thus, for each $s \in S$ we may write 
	$$
	\|(I_A-\varphi)(s)\|_S \leq \|s\|_S + \|\varphi(s)\|_S \leq 1 + \|\varphi\|_S.
	$$
	Subsequently, $\|\im(I_A-  \varphi)\|_S \leq 2 \|\varphi\|_S$. 
	Proposition \ref{centralsubgroup} implies that there exists a surjective group morphism $\map{\pi}{A}{Q}$ such that $\pi(y-x) \notin \pi(\im(I - \varphi))$ and where $$|Q| \leq C \: \max\left\{ \log \left(\|y - x\|_S \right), \|\varphi\|_S\right\} \leq  C \: \max\left\{ \log \left(\|x\|_S + n \right), \|\varphi\|_S\right\}$$ for some $C \in \N$. By construction, it holds that $\pi(y) \notin \pi([x]_\varphi)$. The statements of the theorem then follow immediately. \end{proof}

We still need the following technical result, which is a generalization of Lemma \ref{abelian_image_bound}. 

\begin{lemma}
	\label{norm}
	Let $N$ be a $\mathcal{F}$-group with finite generating subset $S$ and an automorphism $\varphi: N \to N$. There exists a constant $C > 0$ and an integer $k > 0$ such that for every $ y \in N$ with $y \in X_1^\varphi$, it holds that $y =  x \: \varphi(x)^{-1}$ for some $x \in N$ with $$\Vert x \Vert \leq C \max \{ \Vert \varphi \Vert_S, \Vert y \Vert_S \}^k.$$ 
\end{lemma}

\begin{proof}
	We proceed by induction on $j$ such that $y \in X_{c+1-j}^\varphi$. If $j=0$ or thus $y \in X_{c+1}^\varphi = 1$, we can take $x = 1$. Now assume that $j > 0$ and write $c+1-j = i$. Consider the induced map $\bar{\psi}_{\varphi,i}: \faktor{N_i^\varphi}{N_{i+1}^\varphi} \to \faktor{\gamma_i(N)}{\gamma_{i+1}(N)}$. Use Theorem \ref{twisted_pullback_norm_bound} to find constants $C_1, k_1 > 0$ and a generating subset $S^\prime$ for $N_i^\varphi$ such that $\Vert s \Vert_S \leq C_1 \Vert \varphi \Vert_S^{k_1}$ for all $s \in S^\prime$. We apply Lemma \ref{unique_preimage_bound} to find $C_2, k_2 > 0$ such that there exists $x \in N_i^\varphi$ with $$\Vert x \Vert_{S^\prime} \leq C_2\max \{ \Vert \varphi \Vert_S, \Vert y \Vert_S \}^{k_2}$$ and $\bar{\psi}_{\varphi,i} (x \: N_{i+1}^\varphi) = x \: \varphi(x)^{-1} \gamma_{i+1}(N) = y \: \gamma_{i+1}(N)$. Write $z = x \: \varphi(x)^{-1}$. In particular, we get that $$\Vert x \Vert_S \leq C_1 C_2  \Vert \varphi \Vert_S^{k_1} \max \{\Vert \varphi \Vert_S, \Vert y \Vert_S \}^{ k_2} \leq C_1 C_2 \max \{\Vert \varphi \Vert_S, \Vert y \Vert_S \}^{k_1 + k_2}. $$ By construction, $z^{-1} \: y \in X_{i + 1}^\varphi$ with \begin{align*} \Vert z^{-1} y \Vert_S & \leq \Vert y \Vert_S + \Vert x \Vert_S + \Vert \varphi(x) \Vert_S \\  & \leq \Vert y \Vert_S  + C_1 C_2 \max \{\Vert \varphi \Vert_S, \Vert y \Vert_S \}^{k_1 + k_2} + C_1 C_2 \max \{\Vert \varphi \Vert_S, \Vert y \Vert_S \}^{k_1 + k_2 + 1},\end{align*} and so, we can use the inductive hypothesis to finish the proof. 
\end{proof}

For the rest of this section, we fix an automorphism $\map{\varphi}{N}{N}$ and work with the automorphisms $\map{\varphi_x}{N}{N}$ given by $\varphi_x(n) = x \: \varphi(n) \: x^{-1}$. Additionally, we denote the induced automorphism on $\faktor{N}{N^{m}}$ where $m \in \N$ as $\bar{\varphi}_{x,m}$ The subgroup $N_{\varphi_x}$ will be denoted as $N_x$ to simplify notation. Similarly, we denote $\left( \faktor{N}{ N^m}\right)_{\bar{\varphi}_{x,m}}$ via $N_{x,m}$. 

\begin{thm}\label{main_thm}
	Let $N$ be a $\mathcal{F}$-group with a finite generating subset $S$. Let $x \in N$ and $\varphi \in \Aut(N)$. There exist natural numbers $k_1,k_2$, and $k_3$ such that $$\Conj_{N,x,S}^\varphi(n) \preceq  (\|\varphi\|_S)^{k_1} \: \pr{\|x\|_S}^{k_2}  \: n^{k_3}$$ for any $\varphi \in \Aut(N)$. In particular, $\Conj_{N,S}^\varphi(n) \preceq \pr{\|\varphi\|_S}^{k_1}\:n^{k_2 + k_3}$ and $\Tconj_{N,S}(n) \preceq n^{k_1 + k_2 + k_3}$.
\end{thm}  
\begin{proof}
	We proceed by induction on nilpotency class, and observe that the base case is given by Proposition \ref{twisted_abelian_separability}. Thus, we may assume that $c(N) > 1$. 
	
	Let $\varphi \in \Aut(N)$ and fix $x \in N$. Suppose that $y \in \N$ satisfies $x \nsim_\varphi y$ and $\|y\|_S \leq n$. Our goal is to construct a finite quotient $N / K$ such that $\pi_K(y) \notin \pi_{K}([x]_{\tilde{\varphi}})$ and then bound $|N/K|$ in terms of $\|x\|_S$, $\|\varphi\|_S$, and $n$. 
	
	Denote by $\bar{\varphi}$ the automorphism of $\faktor{N}{\ga_{c(N)}(N)}$ induced by $\varphi$. If $x \: \ga_{c(N)}(N) \nsim_{\bar{\varphi}} y \: \ga_{c(N)}(N)$, then the inductive hypothesis implies there exist integers $k_1,k_2, k_3$ and $C_1 > 0$ satisfying the following. There exists a surjective group morphism to a finite group $\map{\pi_1}{\faktor{N}{\ga_{c(N)}(N)}}{Q_1}$ such that $\pi_1(y \: \ga_{c(N)}(N)) \notin [\pi_1( x \: \ga_{c(N)}(N))]_{\bar{\varphi}}$ and where $$|Q_1| \leq  C_1 \: (\|\bar{\varphi}\|_{\bar{S}})^{k_1} \: (\|\bar{x}\|_S)^{k_2} \: n^{k_3}.$$ 
	
Thus, we can assume that $x \: \ga_{c(N)}(N) \nsim_{\bar{\varphi}} y \: \ga_{c(N)}(N)$. Writing $y = y  x^{-1} x$, Lemma \ref{nearly_twisted_conjugate_lemma} implies that $y x^{-1} \in X_1^{\varphi_x} \gamma_c(N)$. By Lemma \ref{norm}, there exists $y_0$ with the norm bounded as described in the lemma and $ y x^{-1} = y_0 \varphi_x(y_0)^{-1} z$ for some $z \in \gamma_c(N)$. Solving for $z$, we get that there exist constants $C_2, k_4$ such that $$\Vert z \Vert_S \leq C_2 \left(   \Vert \varphi_x \Vert_S \: \Vert x \Vert_S \: n \right)^{k_4}.$$ Now we see that $y_0^{-1} y \varphi \left(y_0 \right) =  x  z$ or thus $y \sim_{\varphi} xz$. 
	
	Theorem \ref{twisted_pullback_norm_bound} implies that $\|N_x\|_S\leq C_3 (\|\varphi_x\|_S)^{k_5}$ where $C_3 > 0$ is some constant and $k_5 \in \N$. For each $s \in S$, we may write $\|\varphi_x(s)\|_S = \|x \: \varphi(s) \: x^{-1}\|_S \leq 2\:\|x\|_S + \|\varphi(s)\|_S$. Subsequently,
	$\|\varphi_x\|_S \leq 2 \: \|x\|_S \: \|\varphi\|_S$, and thus, $\|N_x\|_S \leq C_4 \: (\|\varphi\|_S)^{k_5} \: \pr{\|x\|_S}^{k_5}$ where $C_4 = C_3 \: 2^{k_5}$. 
	
	Now fix $k_6 \in \N$ and constant $C_5 > 0$ from Proposition \ref{centralsubgroup} for $N$. Lemma \ref{nearly_twisted_conjugate_lemma} implies that $z \notin N_x$. Therefore, there exists a surjective group morphism $\map{\pi_2}{N}{Q_2}$ such that $\pi_2(z) \notin \pi_2(N_x)$ and where 
	\begin{align*}
	|Q_2| &\leq C_5 \:( \max\{\|N_x\|_S, \log \left( \|z\|_S \right)\})^{k_6} \leq C_5 \left(\|N_x\|_S \cdot \|z\|_S\right)^{k_6} \\
& \leq 2^{k_4} \left(C_2 C_4\right)^{k_6} C_5 \Vert \varphi \Vert_S^{(k_4 + k_5) k_6 } \Vert x \Vert_S^{(2k_4 + k_5) k_6} n^{k_4 k_6} .
\end{align*} 
Moreover, we may assume that $\vert Q_2 \vert = p^\al$ where $p$ is a prime. Lemma \ref{congruence_subgroup_of_finite_index_subgroup} implies that $N^{p^{\al}} \leq \ker(\pi_2)$, and subsequently, $\pi_{p^{\al}}(z) \notin \pi_{p^{\al}}(N_x)$. 
	
	Using the natural number $k^\ast(p,c(N)-1)$ and notation from Corollary \ref{twisted_central_pullback_matrix_reduction}, it follows that $\rho_1(N_{x,m}) = \pi_{m}(N_x)$ where $m = p^{\al + v_p(D_{\varphi_x}) + k^\ast(p,c-1)}$. 	We claim that $\pi_{m}(x \: z) \notin \pi_{m}([x]_\varphi)$. Indeed otherwise, $\pi_{m}(z) \in N_{x,m} .$ Thus, $\pi_{p^\al} (z)= \rho_1(\pi_m(z)) \in \pi_{p^\al}(N_x)$ which is a contradiction.

	For the bound on the order, we combine the previous inequalities. Corollary \ref{lemma_2_blackburn_generalization_bound} implies there exists a constant $C_6 > 0$ and an integer $k_7$ such that $p^{k_0 + v_p(D_\varphi)} \leq C_6 \: (\|\varphi\|_S)^{k_7}$.	Therefore,
	\begin{align*}
	|N / N^{m}| &\leq (p^{\al + k_0 + v_p(D_\varphi)})^{h(N)} =  p^{\al \: h(N)} \: \left( C_6 \: (\|\varphi\|_S)^{k_7}\right)^{h(N)},
	\end{align*}
	and thus, the first statement holds because of the bounds on $p^\al$. 
	The last two statements of the theorem follow immediately.
\end{proof}
\section{Virtually Nilpotent Groups}
This section is broken up into two parts. The first part is a technical detour,  whereas the second subsection contains the main results of the section.

\subsection{Separability of Extensions of Twisted Conjugacy Separable Groupss}
We start this subsection with the following definition.
\begin{defn}
	Let $X \subset G$ be a subset (not necessarily finite). Let $\map{\Farb_{G,X,S}}{\N}{\N}$ be defined as $$\Farb_{G,X,S}(n) = \max \{D_{G}(X,g) \: | \: g \in \left( G\setminus  X  \right) \cap B_{G,S}(n) \},$$ where we take $\max \emptyset = 0$. 
\end{defn}

Recall that the function $D_{G}(X,\cdot)$ was introduced on page \pageref{def_reldepfun}. The function $\Farb_{G,X,S}$ measures the complexity to separate elements of $G$ from the set $X$ in finite quotients. The function does not depend on the choice of generating subset.

\begin{lemma}
	Let $G$ be a finitely generated group, and let $X \subset G$ be a subset. Let $S_1$ and $S_2$ is a two finite generating subsets for $G$, then $\Farb_{G,X,S_1}(n) \approx \Farb_{G,X,S_2}(n)$.
\end{lemma}		
The proof is similar to \cite[Lem 1.1]{Bou_Rabee10} (see also \cite[Lem 1.2]{Bou_Rabee11} and \cite[Lem 2.1]{LLM}). We first give some lemmas about the function $\Farb_{G,X,S}$ before coming to our main result.

Let $H$ be a finite index normal subgroup of $G$, and let $X \subset H$ be a separable subset. For any $x \in H \setminus X$, the following lemma relates the complexity of separating $x$ from $X$ in $G$ to the complexity of separating $x$ from $X$ in $H$.

\begin{lemma}\label{rf_extension}
	Let $G$ be a finitely generated group, and let $H$ be a finite index normal subgroup. Suppose that $S_1$ and $S_2$ are finite generating subsets for $G$ and $H$ respectively. Suppose that $X \subseteq H$ is a separable subset. Then $\Farb_{G,X,S_1}(n) \preceq (\Farb_{H,X,S_2}(n))^{[G : H]}$.
\end{lemma}
\begin{proof}
	Since $H$ is a finite index subgroup, $H$ is an undistorted subgroup. Subsequently, there exists $C > 0$ such that $\|x\|_{S_2} \leq C \: \|x\|_{S_1}$ for all $x \in H$. Let $x \in G$ such that $x \in \left(G\setminus  X\right) \cap B_{G,S_1}(n)$. Suppose that x is not an element of H. Since $X \subseteq H$, we may pass to the quotient $\faktor{G}{H}$ which is a finite group by assumption. Thus, we may assume that $x \in H \setminus X$. 

	Note that $\Vert x \Vert_{S_2} \leq C n$. Hence, there exists a surjective group morphism $\map{\pi}{H}{Q}$ such that $\pi(x) \notin \pi(X)$ where $|Q| \leq \Farb_{H,X,S_2}(C n).$ Since finite groups are linear, we have a finite group $\tilde{Q} \supset Q$ and an induced morphism $\map{\tilde{\pi}}{G}{\tilde{Q}}$ such that $\tilde{\pi}$ restricted to $H$ is equal to our original group morphism $\pi$, and moreover, $|\tilde{Q}| \leq |Q|^{|G:H|}$. Thus, $|\tilde{Q}| \leq \Farb_{H,X,S_2}(C n))^{|G:H|}$, and subsequently, $\Farb_{G,X,S_1}(n) \preceq (\Farb_{H,X,S_2}(n))^{|G:H|}$.
\end{proof}

For separable subsets $\set{X_i}_{i=1}^k$ and $x \in G \setminus \cup_{i=1}^k X_i$, this next lemma relate the complexity of separating $x$ from $\cup_{i=1}^k X_i$ to the complexity of separating $x$ from each $X_i$ individually.
\begin{lemma}\label{rf_union}
	Let $G$ be a finitely generated group with a finite generating subset $S$, and let $\{X_i\}_{i=1}^k$ be a finite collection of proper separable subsets. If $Y = \cup_{i=1}^k X_i$, then $$\Farb_{G,Y,S}(n) \preceq \prod_{i=1}^{k}\Farb_{G,X_i,S}(n).$$
\end{lemma}
\begin{proof}
	Let $x \in G$ such that $x \in \left( G\setminus Y \right) \cap B_{G,S}(n)$. It follows that $x \in \left( G \setminus  X_i \right) \cap B_{G,S}(n)$. Thus, there exists a surjective group morphism $\map{\pi_i}{G}{Q_i}$ such that $\pi(x) \notin \pi(X_i)$ and $|Q| \leq \Farb_{G,X_i,S}(n)$. Let $K = \cap_{i=1}^k \ker (\pi_i)$. By selection, $\pi_{K}(x) \notin \pi_K (X_i)$ for each $i$, and hence, $\pi_K(x) \notin \pi_K(Y)$. Now $\vert \faktor{G}{K} \vert \leq \prod_{i=1}^k |Q_i| \leq \prod_{i=1}^{k}\Farb_{G,X_i,S}(n).$ We conclude that $\Farb_{G,Y,S}(n) \preceq \prod_{i=1}^k \Farb_{G,X_i,S}(n)$.
\end{proof}

This last lemma relates the complexity of separating an element $x \in G/X$ from $X$ to the complexity of separating $x \: y$ from $X \cdot y$.
\begin{lemma}\label{rf_translation}
	Let $G$ be a finitely generated subgroup with a finite generating subset $S$. Suppose that $X \subset G$ is a separable subset, and let $y \in G$. Then $\Farb_{G,X \cdot y, S}(n) \approx \Farb_{G,X,S}(n)$.
\end{lemma}
\begin{proof}
	We need only show that $\Farb_{G,X,S}(n) \preceq \Farb_{G,X \cdot y, S}(n)$. Suppose that $x \in G$ such that $\|x\|_S \leq n$ and $x \notin X \cdot y$. That implies that $x \: y^{-1} \notin X$. Therefore, there exists a surjective group morphism $\map{\pi}{G}{Q}$ such that $\pi(x \: y^{-1}) \notin \pi(X)$ and $|Q| \leq \Farb_{G,X,S}(\|y\|_S + n) \leq \Farb_{G,X,S}(\|y\|_S \cdot n) $. It follows that $\pi(x) \notin \pi(X \cdot y)$ since right translation is a bijection of $G$. Therefore $\Farb_{G,X \cdot y, S}(n) \preceq \Farb_{G,X,S}(n)$.
\end{proof}

Suppose $G$ is a finite extension of $H$. The following proposition shows that the conjugacy class of any element of $G$ can be written as a finite union of right translates of twisted conjugacy classes of elements in $H$. The following proof follows \cite[Thm 5.2]{Felstyn_1}.
\begin{prop}\label{conjugacy_union}
	Suppose that $G$ is a finite generated group that contains a finite index characteristic subgroup $H$, and let $\set{s_i}_{i=1}^{|G:H|}$ be a set of representatives for the right cosets of $H$. Fix an automorphism $\varphi \in \Aut(G)$ and let $f_{x,i}$ be the automorphism of $H$ induced by conjugation by $x_i = s_i \: x \: \varphi(s_i)^{-1}$. If $\bar{\varphi}$ is the automorphism of $H$ induced by $\varphi$, then
	$$[x]_\varphi = \bigcup_{i=1}^{[G:H]} [1_H]_{f_{x,i} \circ \bar{\varphi}} \cdot \: x_i.$$
\end{prop}
\begin{proof}
	Let $\bar{\varphi}$ be the automorphism of $H$ induced by $\varphi$. We may write  
	\begin{align*}
	[x]_\varphi &= \{ y \:  x \: \varphi(y)^{-1} \: | \: y \in G  \} = \bigcup_{i=1}^{[G:H]} \{z \: s_i \: x \: \varphi(s_i)^{-1} \:\varphi(z)^{-1} \: | \: z \in H  \} \\
	&= \bigcup_{i=1}^{[G:H]}\{z \: x_i \: \varphi(z)^{-1} \: x_i^{-1} \: x_i \: | \: z \in H \} = \bigcup_{i=1}^{[G:H]} \{z \: f_{x,i}(\varphi(z))^{-1} \: x_i \: | \: z \in H \} \\
	&= \bigcup_{i=1}^{[G:H]} \{z \: f_{x,i}(\varphi(z))^{-1} \: | \: z \in H \} \cdot x_i = \bigcup_{i=1}^{[G:H]} [1_H]_{f_{x,i} \circ \bar{\varphi}} \cdot \: x_i.
	\end{align*}\end{proof}

Suppose that $H$ is a twisted conjugacy separable group and that $G$ is a finite extension. Additionally, assume that $\varphi \in \Aut(G)$ and $x \in G$. The following theorem relates the quantification of the $\varphi$-twisted conjugacy class of $x$ in $G$ with the quantification of $\psi_i$-twisted conjugacy separability of $H$ where $\psi_i$ are a finite fixed collection of automorphisms of $H$, depending both on $\varphi$ and $x$.

\begin{thm}\label{ext_twisted_conj}
	Suppose that $G$ is a finite generated group that contains a finite index characteristic subgroup $H$, and  let $\set{s_i}_{i=1}^{|G:H|}$ be a set of representatives for the right cosets of $H$. Fix an automorphism $\varphi \in \Aut(G)$ and a $x \in G$. Let $f_{x,i}$ be the automorphism of $H$ induced by conjugation by $x_i = s_i \: x \: \varphi(s_i)^{-1}$. If $S_G$ and $S_H$ are finite generating subsets of $G$ and $H$ respectively, then 
$$\Conj_{G,x,S_G}^\varphi(n) \preceq \prod_{i=1}^{[G:H]}\pr{\Conj_{H,1,S_H}^{f_{x,i} \circ \varphi}(n)}^{[G:H]}.$$\end{thm}
\begin{proof}
	For simplicity in the following arguments, let $X_i = [1_H]_{f_{x,i} \circ \varphi} \cdot \: x_i$ and $X_i' = [1_H]_{f_{x,i} \circ \varphi} $. 
	
	Lemma \ref{rf_translation} implies that $\Farb_{G,X_i,S_G}(n) \approx \Farb_{G,X_i',S_G}(n)$, and since $X_i' \subseteq H$, Lemma \ref{rf_extension} implies that $\Farb_{G,X_i',S_G}(n) \preceq (\Farb_{H,X_i',S_H}(n))^{[G:H]}$. Since $X_i' = [1_H]_{f_{x_i} \circ \varphi}$, we have $\Farb_{H,X_i',S_H}(n) = \Conj_{H,1,S_H}^{f_{x,i} \circ \varphi}(n)$. Given that $[x]_\varphi = \cup_{i=1}^{[G:H]}X_i$, Lemma \ref{rf_union} implies that
	$$
	\Conj_{G,x,S_G}^\varphi(n) \preceq \prod_{i=1}^{[G:H]}\Farb_{H,X_i',S_H}(n).
	$$ Taking everything together, we have
	$$
	\Conj_{G,x,S_G}^\varphi(n) \preceq \prod_{i=1}^{[G:H]}\pr{\Conj_{H,1,S_H}^{f_{x,i} \circ \varphi}(n)}^{[G:H]}.$$ \end{proof}
\subsection{Effective Twisted Conjugacy Separability of Virtually Nilpotent Groups }
We now apply Theorem \ref{ext_twisted_conj} to the context of virtually nilpotent groups to get the first of the two main results of this section.
\begin{thm}\label{virtual_upper_bound}
	Suppose that $G$ is a virtually nilpotent group, and suppose $S$ is a finite generating subset of $G$. For $\varphi \in \Aut(G)$ and $x \in G$, there exists a natural numbers $k_1,k_2,k_3$ such that 
	$$
	\Conj_{G,x,S}^{\varphi}(n) \preceq  \pr{\|\varphi\|_S}^{k_1} \: \pr{\|x\|_S}^{k_2} 
	\: n^{k_3}.
	$$
	In particular, $\Conj_{G,S}^{\varphi}(n) \preceq \pr{\|\varphi\|_S}^{k_1} \: n^{k_2 + k_3}$ and $\Tconj_{G,S}(n) \preceq n^{k_1 + k_2 + k_3}$.
\end{thm}
\begin{proof} 
	If $G$ is finite, then the theorem is clear. Thus, we may assume that $G$ is infinite. We can assume that $N$ is a characteristic $\mathcal{F}$-subgroup of $G$. Let $S'$ be a finite generating subset for $N$ and $\set{s_i}_{i=1}^{|G:N|}$ be a set of right coset representatives of $N$ in $G$. 

	Consider the automorphism $f_{x,i}$ of $N$ induced by conjugation by $s_i \: x \: \varphi(s_i)^{-1}$. Theorem \ref{ext_twisted_conj} implies
	$$\Conj_{G,x,S}^{\varphi}(n) \preceq \prod_{i=1}^{[G:N]}\pr{\Conj_{N,1,S'}^{f_{x,i} \circ \bar{\varphi}}(n)}^{[G:N]}.$$ Theorem \ref{main_thm} implies there exist natural numbers $k_1,k_2,k_3$ such that 
	$$
	\Conj_{H,1,S_H}^{f_{x,i} \circ \varphi}(n) \preceq \pr{\|f_{x,i} \circ \bar{\varphi}\|_{S'}}^{k_1} \: \pr{\|1\|_{S'}}^{k_2} \: n^{k_3} = \pr{\|f_{x,i} \circ \bar{\varphi}\|_{S'}}^{k_1} \:  \: n^{k_3}.
	$$
	Thus, to finish, we give a bound for $\|f_{x,i} \circ \bar{\varphi}\|_{S'}$.
	
Since the subgroup $N$ is undistorted, it suffices to find a bound on $\|f_{x,i} \circ \varphi \|_{S} \leq \|f_{x,i}\|_{S} \cdot \|\varphi\|_{S}$. Note that $f_{x,i}$ is conjugation by $s_i \: x \: \varphi(s_i)^{-1}$, and thus, 
$$\|f_{x,i}\|_{S} \leq 2 \Vert s_i \: x \: \varphi(s_i)^{-1} \Vert_S + 1 \leq 2( \Vert s_i \Vert_S)^2 \: \Vert x \Vert_S \: \Vert \varphi \Vert_S + 1 \leq C \Vert x \Vert_S \Vert \varphi \Vert_S$$ for some $C \in \N$ since there are only finitely many $s_i$. 

That implies $\|f_{x,i} \circ \bar{\varphi}\| \leq C \: \|x\|_S \: (\|\varphi\|_S)^2.$ Therefore,
	$$
	\Conj_{H,1,S_H}^{f_{x,i} \circ \varphi}(n) \preceq C^{k_1} \: \pr{\|x\|_S}^{k_1} \: \pr{\|\varphi\|_S}^{2 \: k_1} \: n^{k_3}.
$$
	By taking everything together, we may write
	$$
	\Conj_{G,x,S}^{\varphi}(n) \leq C^{k_1 \: [G:N]^2} \: \pr{\|x\|_S}^{k_1 \: |G:N|^2} \: \pr{\|\varphi\|_S}^{2 \: k_1 \: |G:N|^2} \: n^{k_3 \: |G:N|^2}.
$$ The last two inequalities follow immediately. 
\end{proof}

For this section's last result, we need the following proposition which is similar to \cite[Cor 10.5]{Pengitore_1}).
\begin{prop}\label{heisenb_conj_class}
	Let $\text{H}_3(\Z)$ be the $3$-dimensional integral Heisenberg group with presentation given by $\innp{x, \: y, \: z | \: [x,y] = z, z \text{ central}}$, and let $p$ be any prime. Suppose $\map{\pi}{H_3(\Z)}{Q}$ is a surjective group morphism where $Q$ is a $q$-group for some prime $q$ distinct from $p$, then $\pi(x^p) \sim_{\id} \pi(x^p \: z^m)$ where $m$ is any natural number. 
\end{prop}
\begin{proof}
Write the conjugacy class $[x^p] = \set{x^p \: z^{t p} \: | \: t \in \Z}$. Let $q^k$ be the order of the element $\pi(z)$ in $Q$. Since $\GCD(p,q^k) = 1$, there exists integers $a,b$ such that $a \: p + b \: q^k = 1$. 
We see that $$\pi(x^p) \sim_{\id} \pi(x^p z^{m \:a \: p}) = \pi(x^p z^{m \: a \: p + m \: b \: q^k}) = \pi(x^p z^m).$$\end{proof}

We reproduce the proof of \cite[Prop 13.1]{Pengitore_1}.
\begin{prop}\label{conj_finite_ext}
	Let $G$ and $H$ be conjugacy separable, finitely generated groups, and suppose that $H$ is a subgroup of $G$. If $x,y \in H$ are two non-conjugate elements of $G$, then $D_{N}([x]_{\text{id}},y) \leq D_{G}([x]_{\text{id}},y)$.
\end{prop}
\begin{proof}
Take $\pi: G \to Q$ be surjective morphism with $|Q| = D_{G}([x]_{\id},y)$ and $\pi(y) \notin \pi([x]_{\id})$. The restriction of $\pi$ to the subgroup $H$ separates the conjugacy classes of $x$ and $y$. Moreover, $\vert \pi(H) \vert \leq \vert Q \vert = D_{G}([x]_{\id},y)$.
\end{proof}

We now have the following theorem which gives the asymptotic behavior for the conjugacy separability quantification function for the class of virtually nilpotent groups.
\begin{thm}\label{last_main_result}
	Let $G$ be a virtually nilpotent group with a finite generating subset $S$, and let $x \in G$. There exist $k_1, k_2 \in \N$ such that $\Conj_{G,x,S}(n) \preceq \pr{\|x\|_S}^{k_1} n^{k_2}$. In particular, $\Conj_{G,S}(n) \preceq n^{k_1 + k_2}$. If $G$ is not virtually abelian, then
	$$
	n^{ \left(c(N)- 1 \right) \left(c(N)+1 \right)} \preceq \Conj_{G,S}(n) \preceq n^{k_1 + k_2}
	$$
	where $N$ is any infinite finitely generated nilpotent subgroup of finite index in $G$.
\end{thm}
\begin{proof}
	Note that if $G$ is finite, then the upper bound clearly holds. Therefore, we may assume that $G$ contains a $\mathcal{F}$-group $N$ as a characteristic finite index subgroup. Thus, Theorem \ref{virtual_upper_bound} implies that there exist integers $k_1,k_2$ such that $
	\Conj_{G,x,S}^{\id}(n) \preceq \pr{\|x\|_S}^{k_1} n^{k_2}.$ Since $\Conj_{G,x,S}^{\id}(n) = \Conj_{G,x,S}(n)$ and $\Conj_{G,S}^{\id}(n) = \Conj_{G,S}(n)$, the first two statements are evident.
	
	For the lower bound, we use similar ideas as in \cite[Thm 1.8]{Pengitore_1}. Assuming that $G$ is not virtually abelian, there exists a $\mathcal{F}$-group $N$ that is a finite index characteristic subgroup of $G$ where $c(N) \geq 2$. We construct an infinite sequence of non-conjugate elements $a_t,b_t$ in $N$ that are also not conjugate in $G$ such that $\|a_t\|_{S'}, \|b_t\|_{S'} \approx n_t$ and  $n_t^{\left(c(N) - 1 \right)  \left( c(N) + 1 \right)} \preceq D_{N}([a_t]_{\text{id}},b_t)$ where $S'$ is a finite generating subset of $N$. Proposition \ref{conj_finite_ext} implies that $D_N([a_t]_{\text{id}},b_t) \leq D_G([a_t]_{\text{id}},b_t)$. Since $N$ is a finite index subgroup of $G$, it follows that $N$ is undistorted in $G$. In particular, $\|a_t\|_S, \|b_t\|_S \approx n_t$. At that point, we have the last statement of the theorem.
	
Take elements $x \in \gamma_{c\pr{N}-1}(N), y \in N$ and $z \in Z(N)$ such that $[x,y] = z^{k}$ and $z$ is primitive for some $k \in \N$. In particular $z^{k} \in \ga_{c(N)}(N)$. For every surjective morphism $\map{\pi}{N}{Q}$ to a finite $q$-group with $q > k$ and $\pi(z) \neq 1$, we get that $c(Q) = c(N)$. Indeed, if $c(Q) < c(N)$, then we must have that $\varphi(z^{k}) = 1$, thus leading to $\text{Ord}_Q(\varphi(z)) \mid k$. Since $\gcd(k,\text{Ord}_Q(\varphi(z))) = 1$ by assumption, this leads to a contradiction. In particular we get that $|Q| \geq q^{c(N) + 1}$ in this case.
We observe that the subgroup $\innp{x,y,z^k} \leq G$ is isomorphic to the $3$-dimensional integral Heisenberg group. 
	
	Letting $\{p_t\}_{t=1}^{\infty}$ be an enumeration of primes greater than $\text{max}\{k,[G:N]\}$, we consider the elements $\al_{t,i} = x^{p_t} \: z^{i \: k}$ for $1 \leq i \leq [G:N] + 1$. Since $\al_{t,i}$ are pairwise non-equal central elements of $\faktor{N}{H \cdot N^{p_t}}$, they are pairwise  non-conjugate as elements of $N$.	We claim that there exist $i_0$ such that $\al_{t,1} \nsim_{\text{id}} \al_{t,i_0}$ as elements of $G$. Letting $\set{s_i}_{i=1}^{|G:N|}$ be a set of right coset representatives of $N$ in $G$, we may write the conjugacy class of any element $x \in N$ as $$ [x]_{\id} = \{g x g^{-1} \mid g \in G \} = \bigcup_{i=1}^{[G:N]} \{ n \: s_i \: x \: s_i^{-1 \:} n^{-1} \mid n \in N \} = \bigcup_{i=1}^{[G:N]} \{ n \: x_i \: n^{-1} \mid n \in N \},$$ so as the union of $[G:N]$ conjugacy classes of elements $x_i = s_i \: x \: s_i^{-1}$ in $N$. Since the $[G:N] + 1$ elements $\al_{t,i}$ all lie in different conjugacy classes of $N$, the claim follows. Now take $a_t = \al_{t,1}$ and $b_t = \al_{t,i_0}$.
	
 From \cite[Lem 3.B]{Gromov}, it follows that $\|a_t\|_S,\|b_t\|_S \approx p_t^{1 / (c(N) - 1)}$. We claim that $D_{N}([a_t]_{\text{id}},b_t) \geq p_t^{c(N)+1}$.
	So it suffices to show that for all surjective group morphisms $\map{\pi}{N}{Q}$ where $|Q| < p_t^{c(N)+1}$ that $\pi(a_t) \sim_{\id} \pi(b_t)$. If $\pi(z^k) = 1$, then $\pi(a_t) \sim_{\id} \pi(b_t)$. Thus, we may assume that $\pi(z^k) \neq 1$. By \cite[Thm 2.7]{Hall_notes}, we may assume that $Q$ is a finite $q$-group where $q$ is a prime. 
	From our assumptions on $p_j$ and $\pi$, we know that if $q = p_j$, then $\vert Q \vert \geq  p_t^{c(N)+1}$.
	Hence, we can also assume that $q \neq p_j$. Since $\innp{x,y,z^k}$ is isomorphic to the $3$-dimensional integral Heisenberg group, Proposition \ref{heisenb_conj_class} implies that there exists $g \in \innp{x,y,z^k}$ such that $\pi(g \: a_t \: g^{-1}) = \pi(b_t)$. Therefore, $\pi(a_t) \sim_{\id} \pi(b_t)$ and as explained above this argument ends the proof.	
\end{proof}
\section{Some Examples} 
\label{sec:examples}
In this last section, we work some explicit examples.

\subsection{Heisenberg group}

In this section, we work out the twisted conjugacy seperability function for the discrete Heisenberg group
$H_3 (\mathbb{Z}) = \{ (x,y,z) \mid x,y,z \in \mathbb{Z} \}$ with group law given by $$(x_1,y_1,z_1)(x_2,y_2,z_2) = (x_1+ x_2, y_1+ y_2, z_1 + z_2 + x_1 y_2).$$ Fix the generating subset $S = \{(1,0,0), (0,1,0)\}$.

Any automorphism $\map{\varphi}{H_3(\Z)}{H_3(\Z)}$ can be written as
$$\varphi(x,y,z) = (ax + by, cx + dy, ex + fy + Dz)
$$
where $D = ad - bc = \pm 1$. We let 
$$
A = \begin{bmatrix}
a & b \\
c & d
\end{bmatrix} \in \GL(2,\Z).
$$

Note that if $D = -1$, then $(0,0,2) \in (H_3(\Z))_{\varphi_X}$ for every $X \in H_3(\Z)$. So in this case, $\Vert (H_3(\Z))_{\varphi_X} \Vert_S$ is uniformly bounded over all $X$.

We compute the function $\Conj_{H_3(\Z),S}^\varphi(n)$ for a general automorphism $\varphi$. Let $X = (x_1,y_1,z_1)$ and $Y =(x_2,y_2,z_2)$ be two elements in $H_3(\Z)$ with $\Vert X \Vert_S, \Vert Y \Vert_S \leq n$ such that $X \nsim_\varphi Y$. If $\pi_{\gamma_2(H_3(\Z))}(X) \nsim_\varphi \pi_{\gamma_2(H_3(\Z))}(Y)$, then Proposition \ref{twisted_abelian_separability} (with the automorphism $\varphi$ fixed here) implies that we can seperate the twisted conjugacy classes in a finite quotient of norm $\preceq \log(n)$. So we will always assume that $\pi_{\gamma_2(H_3(\Z))}(X) \sim_\varphi \pi_{\gamma_2(H_3(\Z))}(Y)$. In this way, it also follows that $\log(n) \preceq \Conj_{H_3(\Z),S}^\varphi(n)$.

Note that  $(H_3(\Z))_2^\varphi = (H_3(\Z))_2^{\varphi_X}$ for all $X \in N$.  The group $(H_3(\Z))_2^\varphi$ can have rank $1, 2$ or $3$, corresponding to dimension $0, 1$ or $2$ for the eigenspace of $A$ corresponding to eigenvalue $1$.

\textbf{Case 1: $\rank((H_3(\Z))_2^\varphi) = 1$}

In this case, $(H_3(\Z))_2^\varphi = \gamma_2(H_3(\Z))$, and thus, $N_{\varphi_X} = 1$ That implies that $x_1 = x_2$ and $y_1 = y_2$. If $D = -1$, we get that $\Conj_{H_3(\Z),S}^\varphi(n)$ is bounded (since there are only a finite number of twisted conjugacy classes). If $D = 1$,  then $\Conj_{H_3(\Z),S}^\varphi(n) \preceq (\log(n))^3$ since it corresponds to separating central elements in the group $H_3(\Z)$. By using the elements $(0,0,1)$ and $(0,0,p)$ for increasing primes, it is easy to show that indeed $\Conj_{H_3(\Z),S}^\varphi(n) \approx (\log(n))^3$.

\textbf{Case 2: $\rank(N_2^\varphi) = 2$}

Note that the second eigenvalue of $A$ must be $-1$ and thus $D = -1$. The norm $\Vert (H_3(\Z))_{\varphi_X} \Vert$ is uniformly bounded and hence there is a fixed quotient in which we can seperate the twisted conjugacy classes of $X$ and $Y$. We conclude that in this case  $\Conj_{H_3(\Z),S}^\varphi(n) \approx \log(n)$.

\textbf{Case 3: $\rank((H_3(\Z))_2^\varphi) = 3$}

In this case, $A = I_2$ and $(H_3(\Z))_2^\varphi = H_3(\Z)$. In this case, we can copy the proof of \cite[Thm 1.6]{Pengitore_1} to find that $\Conj_{H_3(\Z),S}^\varphi(n) \approx \Conj_{H_3(\Z),S}^{\text{id}}(n) \approx n^3$. 

Note that in all three cases, we have $\Conj_{H_3(\Z),S}^\varphi(n) \preceq n^3 \approx \Conj_{H_3(\Z),S}(n)$. 
\subsection{A $5$-dimensional Example}
Let $N$ be the nilpotent group given by 
$$
N = \innp{a_1,a_2,a_3,b_1,b_2 \: | \: [a_1,a_2] = b_1, [a_2,a_3] = b_2, \text{ and } b_1,b_2 \text{ central }}
$$
with generating set $S = \{a_1, a_2, a_3, b_1, b_2\}$. Let $\map{\varphi}{N}{N}$ be the automorphism given by 
$$
\varphi : \begin{cases}
a_1 \mapsto a_1 \\
a_2 \mapsto a_2 \\
a_3 \mapsto a_1 \: a_3 \\
b_1 \mapsto b_1 \\
b_2 \mapsto b_1 \: b_2.
\end{cases}
$$ 
Every element $x \in N$ can be uniquely expressed as 
$$
x = a_1^{\al_1} a_2^{\al_2} a_3^{\al_3} b_1^{\beta_1} \: b_2^{\beta_2}
$$
for some $\alpha_i, \beta_j \in \Z$. 

Note that the subgroup $N_2^{\varphi_x} = N_2^\varphi$ for every $x \in N$. It follows immediately that $$N_2^{\varphi_x}  = \innp{ a_1,a_2,b_1,b_2};$$ hence, $\|N_2^{\varphi_x}\|_S = 1$. These facts allow us to prove the following proposition.
\begin{prop}
There exists a constant $C \in \N$ such that for every $x\in N$ with $\|x\|_S \leq n$, it holds that $\|\im \left( \psi_{\varphi_x,2}) \right)\|_S \leq C \: \sqrt{n}$.
\end{prop}
\begin{proof}
	The subgroup $\im \left( \psi_{\varphi_x,2}) \right)$ is generated by the elements $a_1 \: x \: \varphi(a_1)^{-1} \: x^{-1}$, $a_2 \: x \: \varphi(a_2)^{-1} \: x^{-1}$, $b_1 \: x \: \varphi(b_1)^{-1} x^{-1} = 1$ and $b_2 \: x \: \varphi(b_2)^{-1} x^{-1} = b_1^{-1}$. The statement now follows similarly as in Proposition \ref{important_morphism_bound}.
\end{proof}

We now compute the asymptotic upper bound for $\Conj_{N,S}^\varphi(n)$.
\begin{prop}
For the group $N$ and $\varphi:N \to N$ as above, it holds $$\Conj_{N,S}^\varphi(n) \preceq n^{3}.$$
\end{prop}
\begin{proof}
	Let $x,y \in N$ such that $\|x\|_S,\|y \|_S \leq n$ and where $x \nsim_\varphi y$. First assume that $\pi_{[N,N]}(x) \nsim_\varphi \pi_{[N,N]}(y)$. Proposition \ref{twisted_abelian_separability} (with the automorphism $\varphi$ fixed here) implies there exists a constant $C_1 > 0$ such that $$D_{N/[N,N]}([\pi_{[N,N]}(x)]_{\bar{\varphi}},\pi_{[N,N]}(y)) \leq C_1 \log \left(\|\pi_{[N,N]}(x) \|_S + \|\pi_{[N,N]} (y) \|_S \right).$$ In particular, there exists a surjective group morphism $\map{\pi_1}{N / [N,N]}{Q_1}$ such that $\pi_1(y) \notin \pi_1([x]_{\varphi})$ and where $$|Q_1| \leq C_1 \: \|\pi_{[N,N]}(x)\|_S \: n.$$ Since $\|\pi_{[N,N]}(x)\|_{\bar{S}} \leq n$, it follows that $D_N([x]_\varphi,y) \leq  C_1 \: \log(2n)$. 
	
	Thus, we may assume that $y = x \: z$ where $z \in \ga_2(N) = \langle b_1, b_2 \rangle$ and $\|z\|_S \leq C_2 \: n$. Proposition \ref{centralsubgroup} implies that there exists a surjective group morphism $\map{\pi_2}{N}{Q_2}$ such that $\pi_2(z) \notin \pi_2(\im(\psi_{\varphi_x,2}))$ and where 
	$$
	|Q_2| \leq C_2 \: \text{max} \set{\| \im(\psi_{\varphi_x,2})\|_S, \log \pr{\|z\|_S}}^{2 \: \Phi(N)}
	$$
	for some $C_2 \in \N$. We claim that $\pi_2(x \: z) \notin \pi_2([x]_\varphi)$, and for a contradiction suppose otherwise. Thus, there exists $y \in N$ such that $\pi_2(x \: z) = \pi_2(y \: x \: \varphi(y)^{-1}).$ That implies we may write $\pi_2(z) = \pi_2(y \: x \: \varphi(y)^{-1} \: x)$. In particular, $\pi_2(z) \in \im(\psi_{\varphi_x,2})$ which is a contradiction.
	
	By calculation, one can see that $\Phi(N) = 3$. We conclude that for some $C_3 >0$, we have $D_N([x]_\varphi,y) \leq C_3 \max \{ n^3, \log(n)^6 \}$ and thus $\Conj_{N,S}^\varphi \preceq n^3$.
	\end{proof}

\section{Future Questions}

We list some of the open problems which remain after the main results of this paper.

In Theorem \ref{two_step_conjugacy_precise} we computed an upper bound for the conjugacy separability function for nilpotency class $2$. We conjecture that for a general nilpotent group, this function is of the following form.
\begin{conj}
	Let $N$ be a $\mathcal{F}$-group of nilpotency class $c > 1$, then $\Conj_N(n) \approx n^{\Phi(N) (c - 1)}$. 
\end{conj}

The lower bound was already given in \cite{Pengitore_1}, but unfortunately this argument contains a gap. All known examples satisfy these bounds, so it is still a reasonable conjecture to propose.

In Proposition \ref{centralsubgroup}, we computed an upper bound for the effective seperability function in the case of central subgroups. There is not yet a description for the subgroup separability function for general subgroups. That leads to the following question.
\begin{ques}
	Compute the effective subgroup seperability function for finitely generated nilpotent groups.
\end{ques}

In Section \ref{sec:examples} we gave several explicit examples of the function $\Conj_{N,S}^\varphi$ for nilpotent groups $N$. In all these examples, an upper bound was given by $\Conj_{N,S}$, so in the case of the identity map. We conjecture that that is always the case.

\begin{conj}
	Let $N$ be a nilpotent group with generating subset $S$. For every automorphism $\varphi \in \Aut(N)$, it holds that $$\Conj_{N,S}^\varphi(n) \preceq \Conj_{N,S}(n).$$
\end{conj}

For all known examples, the conjugacy function only depends on the rational or even real Mal'cev completion of the $\mathcal{F}$-group $N$.  It is an open question whether that is true in general.

\begin{ques}
	Let $N_1$ and $N_2$ be two (abstractly) commensurable finitely generated virtually nilpotent groups. Is it true that $$\Conj_{N_1}(n) \approx \Conj_{N_2}(n)?$$ 
\end{ques}

\bibliography{bib}
\bibliographystyle{plain}
\end{document}